\documentclass[11pt]{amsart}

\usepackage{epigamath}

\usepackage[english]{babel}

\usepackage{amsmath,amssymb,amsfonts}
\usepackage[english]{babel}
\usepackage{amsxtra}
\usepackage{mathtools}
\usepackage[normalem]{ulem}
\usepackage[all]{xy}
\usepackage{pgf,tikz}
\usetikzlibrary{arrows}

\newtheorem{theorem}{Theorem}[section]
\newtheorem{lemma}[theorem]{Lemma}
\newtheorem{claim}[theorem]{Claim}
\newtheorem{proposition}[theorem]{Proposition}

\theoremstyle{definition}
\newtheorem{definition}[theorem]{Definition}
\newtheorem{construction}[theorem]{Construction}
\theoremstyle{remark}
\newtheorem{remark}[theorem]{Remark}

\numberwithin{equation}{section}

\newcommand{\IC}{\mathbb C}
\newcommand{\IF}{\mathbb F}

\newcommand{\IP}{\mathbb P}
\newcommand{\IQ}{\mathbb Q}
\newcommand{\IR}{\mathbb R}
\newcommand{\IZ}{\mathbb Z}

\newcommand{\cI}{\mathcal{I}}

\newcommand{\cO}{\mathcal{O}}

\newcommand{\HH}{\mathrm{H}}
\newcommand{\BL}{\mathrm{Bl}}

\newcommand{\crep}{\mathrm{a}}
\newcommand{\mult}{\mathrm{mult}}
\newcommand{\sing}{\mathrm{sing}}
\newcommand{\smooth}{\mathrm{sm}}
\newcommand{\klt}{\textit{klt }}

\DeclareMathOperator{\Spec}{Spec}

\DeclareMathOperator{\Mov}{Mov}
\DeclareMathOperator{\NE}{NE}
\DeclareMathOperator{\Exc}{Exc}
\DeclareMathOperator{\Bs}{Bs}
\DeclareMathOperator{\Proj}{Proj}
\DeclareMathOperator{\CL}{Cl}
\DeclareMathOperator{\Aut}{Aut}
\DeclareMathOperator{\NN}{N}
\DeclareMathOperator{\Div}{div}
\DeclareMathOperator{\Bir}{Bir}
\newcommand{\resp}{\textit{resp. }}

\EpigaVolumeYear{5}{2021} \EpigaArticleNr{11} \ReceivedOn{December 23,
2020}
\InFinalFormOn{April 8, 2021}
\AcceptedOn{April 26, 2021}

\title{Divisorial contractions to codimension three orbits}
\titlemark{Divisorial contractions to codimension three orbits}

\author{Samuel Boissi\`ere}
\address{Laboratoire de Math\'ematiques et Applications,
UMR 7348 du CNRS,
B\^atiment H3,
Boulevard Marie et Pierre Curie,
Site du Futuroscope,
TSA 61125,
86073 Poitiers Cedex 9,
France}
\email{samuel.boissiere@math.univ-poitiers.fr}

\author{Enrica Floris}
\address{Laboratoire de Math\'ematiques et Applications,
UMR 7348 du CNRS,
B\^atiment H3,
Boulevard Marie et Pierre Curie,
Site du Futuroscope,
TSA 61125,
86073 Poitiers Cedex 9,
France}
\email{enrica.floris@math.univ-poitiers.fr}

\authormark{S. Boissi\`ere and E. Floris}

\AbstractInEnglish{Let $G$ be a connected algebraic group. We study $G$-equivariant extremal contractions whose centre is a codimension three $G$-simply connected orbit.
 In the spirit of an important result by Kawakita in 2001, we prove that those contractions are weighted blow-ups.}



\TitleInFrench{Contractions divisorielles vers des orbites de codimension 3}

\AbstractInFrench{Soit $G$ un groupe alg\'ebrique connexe.
Nous \'etudions les contractions extr\'emales $G$-\'equi-variantes
dont le centre est une orbite $G$-simplement connexe de dimension 3.
Dans l'esprit d'un r\'esultat important obtenu par Kawakita en 2001, nous d\'emontrons que ces contractions sont des \'eclatements \`a poids.}

\acknowledgement{Enrica Floris is supported by the ANR Project FIBALGA ANR-18-CE40-0003-01.}

\begin{document}


\removeabove{0cm}
\removebetween{0cm}
\removebelow{0cm}

\maketitle

\begin{prelims}

\DisplayAbstractInEnglish

\bigskip

\DisplayKeyWords

\medskip

\DisplayMSCclass

\bigskip

\languagesection{Fran\c{c}ais}

\bigskip

\DisplayTitleInFrench

\medskip

\DisplayAbstractInFrench

\end{prelims}


\newpage

\setcounter{tocdepth}{1}

\tableofcontents


\section{Introduction}

Let $Y$ be a smooth complex projective variety. The determinant $K_Y$ of its tangent vector bundle  is a line bundle canonically attached to~$Y$ and called the {\it canonical divisor}.
The so-called {\it numerical properties} of $K_Y$ determine the geometry of~$Y$. For example, it is well-known that if $Y$ is a smooth surface, then either
$Y$~is covered by rational curves which have negative intersection with the canonical divisor or, by the Castelnuovo theorem, the curves having negative intersection with the canonical divisor are rational and can be contracted, giving a birational morphism:
\[
Y\longrightarrow X,
\]
such that the canonical divisor of $X$ has non-negative intersection with every curve.

The minimal model program (MMP) aims to achieve a similar description for higher-dimensional varieties.
In the development of the theory, it has become necessary to take singular varieties into consideration.
If $Y$ is normal, its canonical divisor can be defined in the following way. Denote by $Y^\smooth$ the non-singular locus of~$Y$. Writing $K_{Y^\smooth}=\sum_i a_i D_i$, where $D_i$ are prime divisors on $Y^\smooth$ and $a_i\in\IZ$, one defines $K_Y\coloneqq\sum_i a_i \overline D_i$, where $\overline D_i$ is the Zariski closure of $D_i$ in $Y$. Since the singular locus $Y^\sing$ of $Y$ has codimension at least two in~$Y$, the divisor~$K_Y$ is the unique divisor extending $K_{Y^\smooth}$. In order to be able to compute intersection numbers, we need the divisor $K_Y$, which is a priori only a Weil divisor, to be a $\IQ$-Cartier divisor. Moreover, we always assume that $Y$ has at worst {\it terminal} singularities (see~\S\ref{ss:MMP}).

The first step of the MMP consists in looking at the curves which have negative intersection with~$K_Y$.
This is achieved by the \emph{cone and contraction theorem}~\cite[Theorem~3.7]{KollarMori}, which describes the cone of numerical equivalence classes of curves in~$Y$.
It states that if $R$~is an extremal ray of the cone, having negative intersection with the canonical divisor, then there is a morphism $f\colon Y\to X$, called an {\it extremal contraction}, contracting exactly those curves whose class belongs to $R$.
There are three possibilities for the morphism $f$ (see~\cite[Proposition~2.5]{KollarMori}):
\begin{itemize}
\item {\it Divisorial contraction}: $\dim Y=\dim X$ and the exceptional locus of $f$ has codimension 1;
\item {\it Small contraction}: $\dim Y=\dim X$ and the exceptional locus of $f$ has codimension $\geq 2$;
\item {\it Mori fibre space}: $\dim Y>\dim X$.
\end{itemize}
Those morphisms are the elementary bricks of the minimal model program and of the Sarkisov program.
The first conjecturally associates to a variety a simpler model (either a minimal model or a Mori fibre space, see~\cite[\S 2.1]{KollarMori})
and the second describes the relation between two different Mori fibre spaces associated to the same variety (see~\cite{Cor95,HMcK}).

In this note, we focus on divisorial contractions. It is indeed useful for many applications to know what a divisorial contraction from a certain variety~$Y$ looks like, as it gives information on the possible outcomes after performing an MMP on~$Y$.
Divisorial contractions from a smooth variety~$Y$ have been studied for instance in~\cite{Ando,Wisn} and extremal contractions from mildly singular varieties in~\cite{And95,And18,AndTas14,AndTas16}.
Here we focus our attention on the singularities of $X$, rather than on the singularities of~$Y$.
If $X$~is a smooth surface and if the centre $Z\coloneqq f(\Exc(f))$ of~$f$ is a point, then by the Castelnuovo theorem, $f$~is a smooth blow-up.
In dimension three, Kawakita proves the following result:

\begin{theorem}[Kawakita, {\cite[Theorem~1.1]{Kawakita}}] \label{th:kawakita}
Let $f\colon Y\to X$ be $3$-dimensional divisorial contraction, which contracts its exceptional divisor to a smooth point. Then $f$~is a weighted blow-up.
\end{theorem}

This result is particularly interesting if $X$ has a Mori fibre space structure $X\to B$  and if one wants to study in detail the Sarkisov program starting from $X/B$. Indeed, for that, it is necessary to know the contractions \textit{from} and \textit{to} $X$.

\medskip

If $G$ is a connected algebraic group acting on $X$, then the Sarkisov program can be used to determine whether $G$ is a maximal subgroup of the group $\Bir(X)$ of birational automorphisms of~$X$ (see~\cite{BFT19,Flo18}).
Motivated by the study of the $G$-equivariant Sarkisov program, we turn our attention to $G$-equivariant divisorial contractions
to a codimension three centre contained in the smooth locus and we prove the following result for any $d\geq 3$.

\begin{theorem}\label{th:main}
Every  $d$-dimensional $G$-equivariant divisorial contraction to a  $G$-simply connected $G$-orbit of dimension~$d-3$,
 contained in the smooth locus of $X$ is a $G$-equivariant weighted blow-up.
\end{theorem}

The strategy of proof of Theorem~\ref{th:main} is based on a construction of Kawakita~\cite{Kawakita} and uses an induction
argument to reduce the statement to Kawakita's theorem~\ref{th:kawakita}. Given a $G$-equivariant divisorial contraction $f\colon Y\to X$ with exceptional divisor $E$, satisfying our assumptions (see~\S\ref{ss:Gdiv}), the starting point in \S\ref{s:tower} is the iterative construction of a second birational model of~$X$, which ends up to a $G$-equivariant birational morphism $h\colon X_n\to X$ contracting an irreducible exceptional divisor~$E_n$. The $G$-equivariant iterative process provides the needed information on the valuation defined by~$E_n$, which corresponds to a weighted blow-up, and by an induction argument in \S\ref{s:proof}, using iterated hyperplane sections, we show that $E$~defines the same valuation on~$X$ as~$E_n$.

Finally, we present in \S\ref{s:counterexample} an example showing that the hypothesis that $Z$ is an orbit is essential. We describe a terminal extraction with one-dimensional center which is not a weighted blow-up.
\medskip

\textbf{Aknowledgements.} The authors warmly thank Andreas H\"oring, J\'anos Koll\'ar, Massimiliano Mella, Andrea Petracci, Erik Paemurru and Ronan Terpereau for helpful comments, enlightening discussions and precious help during the preparation of this work. We also thank the anonymous referee for helpful comments which improved the presentation of this paper.

\section{The setup}

\subsection{Convention} We work over the field $\IC$ of complex numbers.
A \emph{variety} is an integral, separated scheme of finite type over $\IC$. A \emph{divisor} is either Cartier, $\IQ$-Cartier or $\IR$-Cartier, depending on the context.
We denote by $\sim$ (\resp $\sim_{\IQ}$, $\sim_{\IR}$) the linear (\resp $\IQ$-linear, $\IR$-linear) equivalence relation of divisors. Unless explicitely mentioned, all varieties are assumed to be projective.

If $f\colon X\to Y$ is a morphism between two varieties $X$~and~$Y$, and $D_1, D_2$ are divisors on $X$, we write
$D_1\sim_f D_2$
 (\resp $\sim_{\IQ,f}$, $\sim_{\IR,f}$) if there is a Cartier (\resp $\IQ$-Cartier, $\IR$-Cartier) divisor $\delta$ on $Y$ such that $D_1\sim D_2+f^*\delta$ (\resp $\sim_{\IQ f}$, $\sim_{\IR f}$).
We say that a divisor $D$ is $f$-ample (\resp $f$-antiample, $f$-effective) if there exists an ample (\resp antiample, effective) divisor $D'$ such that $D\sim_{f} D'$. A \emph{pair} $(X,\Delta)$ is the data of a normal projective variety $X$ and a $\IQ$-divisor $\Delta$.

\subsection{Divisorial contractions}

\begin{definition}\label{def:divisorial}
A morphism $f\colon Y\to X$ with connected fibers, between normal projective varieties $Y$ and $X$ is called a \emph{divisorial contraction} if it satisfies all the following conditions:
\begin{enumerate}
\item $Y$ is locally $\IQ$-factorial with terminal singularities;
\item the morphism $f$ is birational and its exceptional locus $E$ is a prime divisor;
\item the canonical divisor $K_{Y}$ is $f$-antiample;
\item the morphism $f$ has relative Picard number one.
\end{enumerate}
\end{definition}

Throughout this paper, we consider a $d$-dimensional divisorial contraction:
\[
f\colon Y\to X,
\]
with $d\geq 3$,
which contracts its exceptional divisor~$E$ to its \emph{centre}~$Z$, that is assumed to be a smooth subvariety contained in the smooth locus of~$X$: for short we call $Z$ a \emph{smooth centre}.
Recall that $X$ is also locally $\IQ$-factorial with terminal singularities~\cite[Proposition~3.36 \& Corollary~3.43(3)]{KollarMori} and therefore the singular locus $X^\sing$ of $X$ has codimension at least three~\cite[Lemma~1.3.1]{BS}.
We have a $\IQ$-linear equivalence of $\IQ$-Cartier divisors:
\begin{equation*}
\label{eq1}
K_{Y}\sim_{\IQ} f^\ast  K_X+aE,
\end{equation*}
where the positive rational number $a\coloneqq\crep(E,X)$ is the discrepancy of~$E$ with respect to~$X$.

\subsection{Equivariant divisorial contractions}\label{ss:Gdiv}

\begin{definition}\label{def:Gdivisorial}
Let $G$ be a connected algebraic group. A divisorial contraction $f\colon Y\to X$ is called \emph{$G$-equivariant} if it satisfies the following conditions:
\begin{enumerate}
\item $X$ and $Y$ are endowed with a regular action of $G$;
\item the contraction $f$ is $G$-equivariant.
\end{enumerate}
\end{definition}

We still denote by $E$ the exceptional divisor of the $G$-equivariant contraction $f\colon Y\to X$ and by $Z$ its centre. The action of $G$ on $X$ induces a regular action on $Z$.
The variety~$Z$ is said $G$-\emph{simply connected} if its $G$-equivariant fundamental group $\pi_1^G(Z)$ is trivial (see for instance \cite{Huisman,Looijenga}).
This implies that every connected, \'etale, $G$-equivariant morphism with target $Z$ is an isomorphism.

\begin{remark}\text{}
\begin{enumerate}
\item If the action of $G$ on $Z$ is transitive, then $Z$ is reduced and smooth, and every finite $G$-equivariant morphism is \'etale.
\item If $Z$ is a Fano manifold, it is automatically $G$-simply connected since every \'etale  cover of a Fano variety is an isomorphism (see \cite[Corollary 4.18(b)]{Debarre}). Indeed, take an \'etale cover ${\eta\colon Z'\to Z}$.
The Euler characteristic of $Z'$ is given by:
\[
\chi(Z')=(\deg \eta)\chi(Z),
\]
and $Z'$ is also a Fano manifold since $K_{Z'}=\eta^*K_Z$.
By the Kodaira vanishing theorem, the Euler characteristic of a Fano manifold is one, so $\eta$~is an isomorphism.
\item If $G$ is a semi-simple linear group, then its closed orbits are Fano manifolds by \cite[Corollary 2.1.7]{AGV}.
\end{enumerate}
\end{remark}

To compare with Theorem~\ref{th:main}, which deals with codimension $3$ orbits, the understanding of codimension $2$ orbits is an easy consequence of a result due to Ando~\cite{Ando}:

\begin{proposition}\label{codim2}
Every  $d$-dimensional $G$-equivariant divisorial contraction to a $(d-2)$-dimensional  $G$-orbit contained in the smooth locus of $X$ is a blow-up.
\end{proposition}

\begin{proof}
 Let $f\colon Y\to X$ be a $G$-equivariant divisorial contraction with centre $Z$ and exceptional divisor $E$. We first show that the singular locus $Y^\sing$ of $Y$ does not meet~$E$. Otherwise, $Y$ would have a singularity over $Z$, but $f$ is $G$-equivariant and the action of $G$ preserves the singularities, so we would get $f(Y^\sing)\supseteq Z$. Since $Z$ has codimension~$2$, this would imply that $Y^\sing$ has a component of codimension~$2$: this is impossible since $Y$~has terminal singularities. So $Y$ is smooth in a neighbourhood of $E$ and $f_{|E}\colon E \to Z$ is equidimensional since $f$ is $G$-equivariant.
We can thus apply \cite[Theorem~2.3]{Ando}, showing that $f$ is a blow-up.
\end{proof}

Our assumption of a $G$-action might look strong, however we believe that it is necessary. In Section~\ref{s:counterexample} we provide an example of an extremal contraction
in dimension four, non equivariant, which is not a weighted blow-up.


\section{Preliminaries on the MMP with scaling}\label{ss:MMP}

We recall for further use some notions on the Minimal Model Program (MMP), following the terminology and notation of Koll\'ar--Mori~\cite{KollarMori}.

\begin{definition}
 A pair $(X,\Delta)$ is called \klt if $K_X+\Delta$ is $\IQ$-Cartier, $\lfloor\Delta\rfloor=0$ and there exists a log resolution $\mu\colon\hat X\to X$ (see~\cite[Notation 0.4~(10)]{KollarMori})
 such that:
\[
K_{\hat X}+\mu_*^{-1}\left(\Delta\right)\sim_\IQ\mu^*(K_X+\Delta)+\sum_E a_E E, \text{ with }  a_E>-1 \text{ for all } E,
\]
where the sum runs over the exceptional divisors of~$\mu$.
A $\IQ$-factorial variety $X$ is called  \emph{terminal} if there exists a log resolution $\mu\colon\hat X\to X$
 such that:
\[
K_{\hat X}\sim_\IQ\mu^*K_X+\sum_E a_E E, \text{ with }  a_E>0 \text{ for all } E.
\]
where the sum runs over the exceptional prime divisors of~$\mu$.
\end{definition}

\begin{definition} Let $\widetilde W$, $W$ be two varieties together with a projective morphism $f\colon\widetilde W\to W$.
 A  $\mathbb Q$-divisor $D$ on $\widetilde W$ is \emph{movable} over $W$ if there is a positive integer $m$ such that the intersection of the base locus of $mD$ with every fibre of $f$ has codimension at least two in the fibre.
An $\mathbb R$-divisor $D$ is movable if there are $r_1,\ldots,r_k\in\mathbb R_{\geq0}$ and  movable $\mathbb Q$-divisors $D_1,\ldots,D_k$ such that $D\sim_\IR\sum_{i=1}^k r_i D_i$.
\end{definition}

The sum of two movable $\mathbb Q$-divisors is movable, and therefore numerical equivalence classes of movable divisors form a cone $\Mov(\widetilde W/W)\subseteq \NN^1(\widetilde W/W)$,
 which is in general neither open nor closed.

\begin{definition}
 The \emph{movable cone} $\overline{\Mov}(\widetilde W/W)$ of $\widetilde W$ over $W$ is the closure of $\Mov(\widetilde W/W)$ in $\NN^1(\widetilde W/W)$ with respect to the euclidean topology.
\end{definition}

\begin{remark}\label{rkmov2}
 If $D\in\overline{\Mov}(\widetilde W/W)$, then for every family of curves $\{\Gamma_t\}_t$ such that $\cup_t \Gamma_t$ has codimension one in a fibre of $f$, we have $D\cdot \Gamma_t\geq 0$ for all $t$.
\end{remark}

We recall the procedure called the \textit{Minimal Model Program with scaling of an ample divisor}.
The hypotheses here are stronger than the usual ones, but they are exactly what we need in the sequel.

\begin{construction}[Minimal Model Program with scaling]\label{MMPwsc}
(See \cite[\S 3.10]{BCHM}.)
Let $(\widetilde W,\Delta)$ be a \klt pair such that $\Delta$
is a $\mathbb{Q}$-divisor, and let $f\colon\widetilde W\to  W$  be a proper birational morphism such that $K_{\widetilde W}+\Delta$ is $f$-antiample.
Let $A$ be an ample $\mathbb{Q}$-divisor on $\widetilde W$ such
that $(\widetilde W,\Delta+A)$ is \klt and $K_{\widetilde W}+\Delta+A$ is $f$-nef.
We set:
\[
\lambda_0\coloneqq\inf\{t\in\mathbb{R}_{\geq0}\vert \; K_{\widetilde W}+\Delta+tA\;{\rm is}\;{\rm nef}\;{\rm over}\;  W\}.
\]
Then either $\lambda_0=0$ and $K_{\widetilde W}+\Delta$ is nef, or there is an extremal ray $R_0\in\overline{\NE}(\widetilde W/W)$ such that:
\[
(K_{\widetilde W}+\Delta)\cdot R_0<0 \text{ and  } (K_{\widetilde W}+\Delta+\lambda_0 A)\cdot R_0=0.
\]
If $K_{\widetilde W}+\Delta$ is $f$-nef, or if $R_0$~defines a Mori fibre space, we stop.
Otherwise $R_0$ gives a divisorial contraction  or a flip $\varphi_0\colon\widetilde W\dasharrow W_1$.
We recall that flips exist by~\cite[Corollary 1.4.1]{HMcK10}.
Let $\Delta_1$ and $A_1$ be the strict transforms of $\Delta$ and $A$.

We prove that $K_{ W_1}+\Delta_1+\lambda_0 A_1$ is $f_1$-nef.
Let $(p,q)\colon \widehat W\to \widetilde W\times W_1$ be a resolution of the indeterminacies of $\varphi_0$ (if $\varphi_0$ is a morphism, then $p$ is the identity).
By the negativity lemma~\cite[Lemma~3.39]{
KollarMori} we have:
\[
p^*(K_{\widetilde W}+\Delta+\lambda_0 A)=q^*(K_{ W_1}+\Delta_1+\lambda_0 A_1),
\]
so the result follows from the projection formula.

We further continue with $K_{W_1}+\Delta_1$ and $A_1$.
We finally obtain a sequence of divisorial contractions
and flips:
 \[
\xymatrix{(\widetilde W,\Delta) \eqqcolon (W_0,\Delta_0)\ar@{-->}[r]^-{\varphi_0} \ar[drr]_{f\eqqcolon f_0} &(W_1,\Delta_1)
\ar@{-->}[r]^-{\varphi_1} \ar[dr]^{f_1}& \cdots\ar@{-->}[r]^-{\varphi_{i-2}}& (W_{i-1},\Delta_{i-1})\ar@{-->}[r]^-{\varphi_{i-1}}\ar[dl]^{f_{i-1}} & \cdots\\
& &  W}
\]
At each step $i$, we set:
\[
\lambda_i\coloneqq\inf\{t\in\mathbb{R}_{\geq0}\vert \; K_{W_i}+\Delta_i+tA_i\;{\rm is}\;{\rm nef}\;{\rm over}\; W\},
\]
where $A_i$ (\resp $\Delta_i$) is the push-forward of $A$ (\resp $\Delta$) on $W_i$.
It follows from the definition that $0\leq \lambda_i\leq 1$ for all $i$ and that the sequence $(\lambda_i)_i$ is decreasing. We call this contruction the $(K_{\widetilde W}+\Delta)$\emph{-MMP with scaling of $A$ over $W$}.
\end{construction}

\begin{remark}\label{rkmmp}\text{}
\begin{enumerate}
 \item  In the above construction, if there is a flip, then the strict transform of $A$ is not ample anymore.
 Indeed, assume for instance that $\varphi_0$ is a flip and let $\phi\colon W_0\to Y$ and $\phi^+\colon W_1\to Y$ be the two small contractions involved.
 Then we know that:
 \begin{itemize}
  \item $K_{ W_0}+\Delta+\lambda_0 A$ is $\phi$-trivial and $K_{ W_1}+\Delta_1+\lambda_0 A_1$ is $\phi^+$-trivial;
  \item $K_{W_0}+\Delta$ is $\phi$-antiample and $K_{ W_1}+\Delta_1$ is $\phi^+$-ample,
 \end{itemize}
therefore $A_1$ is $\phi^+$-antiample.

\item\label{rkmmp2} Isomorphisms in codimension one preserve the movable cone.
Indeed, let $\varphi\colon W_0\dasharrow W_1$ be an isomorphism in codimension one.
It is enough to prove that for any movable $\mathbb Q$-divisor  $D$, its pushforward $\varphi_*D$ is a movable  $\mathbb Q$-divisor.
Let $(p,q)\colon \widehat W\to \widetilde W\times W$ be a resolution of the indeterminacies.
Since $\varphi$ is an isomorphism in codimension one, we have $\Exc(p)=\Exc(q)$ and for any $m$ such that $mD$ is Cartier we have:
\[
\Bs(|mp^*D|)\subseteq p^{-1}\Bs(|mD|).
\]
By definition, $\varphi_*D=q_* p^*D$ and  $\Bs(|mq_*p^*D|)=\Bs (q_*|mp^*D|)$.
An irreducible component of $p^{-1}\Bs(|mD|)$ has either codimension at least 2 or is contained in the exceptional locus of $q$.
Then:
\[
\Bs(|mq_*p^*D|)=\Bs (q_*|mp^*D|)=q(\Bs(|mp^*D|))\subseteq q(p^{-1}\Bs(|mD|)),
\]
and this last set has codimension two in $W_1$.
\end{enumerate}
\end{remark}

The following result is an easy termination lemma known to experts. We present here a proof for the reader's convenience, following closely~\cite[Theorem~2.3]{Fujino} (see also \cite[Corollary 1.4.2]{BCHM}).

\begin{lemma}\label{easyterm}
In the MMP with scaling of an ample divisor as above, there is no infinite sequence of flips.
\end{lemma}

\begin{proof}
We prove the statement by contradiction. Assume that $n$ is an integer such that $\varphi_{i}$ is a flip for every $i\geq n$.

{\bfseries Step 1.} We prove first that $K_{W_i}+\Delta_i\not\in\overline{\Mov}(W_i/ W)$ for every $i$.
 Indeed $-(K_{W_i}+\Delta_i)$ is the pushforward of $-(K_W+\Delta)$. This divisor is ample and we can choose an effectif divisor $D$ such that
 $-(K_W+\Delta)\sim_{\mathbb{R},f} D$ and such that the support of~$D$ is not contained in the exceptional locus of $\varphi_{i-1}\circ\cdots\circ\varphi_1$. Denote by~$D_i$ the pushforward of $D$ on $W_i$.
 Then $-(K_{W_i}+\Delta_i)\sim_{\mathbb{R},f} D_i$. The divisor~$D_i$ is effective and non-zero.
Thus, for any family of curves $\{\Gamma_t\}_t$ such that $\cup_t\Gamma_t$ has codimension one in a fibre of~$f_i$ and which is not contained in $D_i$, we have $D_i\cdot\Gamma_t>0$, and by Remark~\ref{rkmov2} we have $K_{W_i}+\Delta_i\not\in\overline{\Mov}(W_i/ W)$. In particular, $K_{W_n}+\Delta_n\not\in\overline{\Mov}(W_n/ W)$.

 {\bfseries Step 2.} Let $\lambda\coloneqq\lim_{i\to\infty} \lambda_i$. We prove that we can assume $\lambda=0$.
 Indeed, if $\lambda>0$, we pick $\Delta'\sim_{\mathbb{R}}\lambda A$ such that the pair $(\widetilde W,\Delta+\Delta')$ is \klt and we run the $(K_{\widetilde W}+\Delta+\Delta')$-MMP with scaling of $A$ over $ W$.
 We notice that each step of the $(K_{\widetilde W}+\Delta+\Delta')$-MMP with scaling of $A$ is a step of the $(K_{\widetilde W}+\Delta)$-MMP with scaling of $A$ over $ W$.
 Moreover:
\[\mu_i\coloneqq\inf\{t\in\mathbb{R}_{\geq0}\vert \; K_{W_i}+\Delta_i+\Delta'_i+tA_i\;{\rm is}\;{\rm nef}\;{\rm over}\;  W\}=\lambda_i-\lambda,
\]
so $\lim_{i\to\infty} \mu_i=0$. we can thus assume that $\lambda=0$.

 {\bfseries Step 3.}
Take for any $i\geq n$ an $f_i$-ample $\mathbb{Q}$-divisor $G_i$ on $W_i$ such that:
\[
\lim_{i\to\infty} G_{i,n}=0\in \NN^1( W_n/W),
\]
where $ G_{i,n}$ is the
strict transform of $G_i$ on $W_n$.
For this, fix an euclidean norm~$\|\cdot\|$  on $ \NN^1( W_n/W)$ and let $A_i$ be an $f_i$-ample $\mathbb{Q}$-divisor on $W_i$ and $A_{i,n}$ its strict transform.
Then put $G_{i,n}=\frac{A_{i,n}}{i\|A_{i,n}\|}.$
We note that the divisor:
\[
K_{W_i}+\Delta_i+\lambda_i A_i+G_i
\]
 is ample over $W$ for every $i$, since by construction $K_{W_i}+\Delta_i+\lambda_i A_i$ is nef.
By Remark~\ref{rkmmp}(\ref{rkmmp2}), the strict transform:
\[
K_{ W_n}+\Delta_n+\lambda_i A_n+G_{i,n}
\]
is
movable on~$W_n$ for every~$i$. Thus $K_{ W_n}+\Delta_n$ is a limit
of movable $\mathbb{R}$-divisors in $\NN^1(W_n/ W)$ and therefore $K_{ W_n}+\Delta_n\in \overline{\Mov}(W_n/ W)$: this contradicts Step~1, so there is no infinite sequence of flips.
\end{proof}

Lemma \ref{easyterm} says that, in the setup of Construction \ref{MMPwsc}, every MMP with scaling of an ample divisor terminates and, in particular, minimal models exist.

\section{Preliminaries on equivariant birational transformations}\label{s:equiv}

We gather some classical constructions on  $G$-equivariant birational transformations that will be needed in the sequel. We consider  a variety  $X$ and a connected algebraic subgroup $G$ of $\Aut^0(X)$.

\subsection{Equivariant blow-ups}
Let $Z$ be a $G$-invariant subvariety of $X$.
Then $G$ acts on the ideal sheaf $I_Z$ of $Z$, so there is a natural action of $G$ on the blow-up $\BL_Z X\coloneqq\Proj\left(\oplus_{k\geq 0} I_Z^{k}\right)$ of $Z$ on $X$ and the natural morphism $\BL_Z X\to X$ is $G$-equivariant.

\subsection{Equivariant weighted blow-ups}\label{ss:wbl}

\begin{definition}
 The weighted blow-up of $X$ of weights $\omega=(\omega_1,\ldots,\omega_r)$, where the positive integers $\omega_i$ have no common divisor, with smooth centre~$Z$ of codimension~$r$, contained in the smooth locus of $X$ is the normalisation of the projectivisation:
\[
\BL_Z^\omega X\coloneqq\Proj\left(\cO_X\oplus\bigoplus_{k\geq 1}\cI_k\right)\to X,
\]
where the sheaves $\cI_k$ are ideal sheaves on $X$ such that, locally analytically on~$X$, there
are smooth coordinate functions $x_1,\ldots,x_r$ defining $Z$,
and for every $k$, the sheaf $\cI_k$ is generated by the monomials $x_1^{a_1}\cdots x_r^{a_r}$ with $\omega_1 a_1+\cdots+ \omega_r a_r\geq k$. A weighted blow-up is called $G$-equivariant if $g(\cI_k)=\cI_k$ for every $g\in G$ and $k\geq 1$. This means that $g^\ast x_1,\ldots,g^\ast x_r$ are smooth coordinate functions defining locally $Z$, for any $g\in G$.
In this case, the morphism $\BL_Z^\omega X\to X$ is $G$-equivariant.
\end{definition}

In the sequel, we consider the weighted blow-up of a $d$-dimensional variety~$X$ along a codimension~$r=3$ smooth centre~$Z$ contained in the smooth locus of $X$, with weights $\omega=(n,m,1)$ such that $n$ and $m$ are coprime. From now on, we assume that $X=\Spec \IC[x_1,\ldots,x_d]$ and that $Z$ is given by the equations $x_1=x_2=x_3=0$. Then $\BL_Z^\omega X=\Proj\left(\bigoplus_{k\geq 0} I_k\right)$,
where $I_k\subset\IC[x_1,\ldots,x_d]$ is the ideal generated by the monomials $x_1^{a_1}x_2^{a_2}x_3^{a_3}$ with $na_1+ma_2+a_3\geq k$. We can define a surjective map of graded rings:
\[
\Phi\colon\IC[x_1,\ldots,x_d][u,v,w]\to \bigoplus_{k\geq 0} I_k,\quad u\mapsto x_1,\quad v\mapsto x_2,\quad w\mapsto x_3,
\]
where, on the left hand side, the degrees are $\deg(x_i)=0$ for all $i$, $\deg(u)=n$, $\deg(v)=m$ and $\deg(w)=1$. Then $\BL_Z^\omega X$ is isomorphic to the proper closed subscheme of $\IC^d\times\IP(n,m,1)$ defined by the weighted homogeneous polynomials which generate the kernel of $\Phi$, that is:
\begin{equation}
\label{eq:8}
x_1^m v^n=x_2^n u^m,\quad x_1 w^n=x_3^n u,\quad x_2 w^m=x_3^m v.
\end{equation}
This shows that the weighted blow-up may be equally defined as the normalization of the closure of the image of the map:
\[
\left(\IC^d\setminus Z\right)\to \left(\IC^d\setminus Z\right)\times\IP(n,m,1),\quad (x_1,\ldots,x_d)\mapsto (x_1,\ldots,x_d,[x_1:x_2:x_3]).
\]

The exceptional divisor $E$ of the weighted blow-up is isomorphic to the weighted projective space $\IP(n,m,1)$.
The valuation $v_\omega\colon \IC(x_1,\ldots,x_d)\setminus\{0\}\to\IZ$ associated to this weighted blow-up is characterized by:
\[
 v_\omega(x_1)=n,\quad v_\omega(x_2)=m,\quad v_\omega(x_3)=1,\quad v_\omega(x_i)=0\quad \forall i\geq 4.
\]

\subsection{Equivariant resolution of singularities}\label{ss:Gres}
The singular locus of $X$ is a $G$-invariant subvariety of $X$.
By \cite[Proposition~3.9.1, Theorem~3.35 \& Theorem~3.36]{Kollar1}
there is a smooth variety $\widetilde X$ together with a regular action of $G$ on $\tilde X$
and a $G$-equivariant birational morphism $\widetilde X\to X$.

\subsection{Equivariant MMP} \label{ss:GMMP}

Any MMP on $X$ is automatically $G$-equivariant. To see it, consider the first step  $X\dasharrow X_1$ of an MMP:
\begin{itemize}
\item If it is an extremal contraction, this is a consequence of the  Blanchard
lemma~\cite[Proposition~4.2.1]{BSU}. We present here an alternative proof. The group $G$ acts trivially on the extremal rays contained in the $K_X$-negative part of the Mori cone, since these rays are discrete.
 Then the extremal ray corresponding to the contraction $X\to X_1$ is $G$-invariant and so is the locus spanned by it, so $X_1$ inherites an action of $G$ making the contraction $G$-equivariant.

\item If it is a flip, given by the composition of two small contractions $\mu\colon X\to Y$
 and $\mu^+\colon X_1\to Y$,
 by the discussion above, there is a $G$-action on $Y$ such that $\mu$ is $G$-equivariant.
 Moreover we have:
\[
X_1\cong \Proj_Y\left( \bigoplus_m\mu_*\mathcal{O}_X(m(K_X))\right).
\]
Since $K_X$ is $G$-invariant, the group $G$ acts on $\mathcal{O}_X(mK_X)$  for any $m$ such that $mK_X$ is Cartier, and subsequently on $X_1$.
\end{itemize}


\section{The tower}\label{s:tower}

Let us recall the setup. We consider a $d$-dimensional $G$-equivariant divisorial contraction $f\colon Y \to X$ with exceptional divisor~$E$, as in Definitions~\ref{def:divisorial} \& \ref{def:Gdivisorial}.
We further assume that the center $Z\subset X$ of the contraction is smooth and contained in the smooth locus of $X$.


\subsection{Construction of the tower}\label{ss:tower}

Starting from $X_0\coloneqq X$ and $Z_0\coloneqq Z$,
we construct inductively the following objects: $G$-equivariant birational morphisms $g_i\colon X_i\to X_{i-1}$,
integral closed subschemes $Z_i\subset X_i$ and prime divisors $E_i$ on $X_i$. This construction follows the same lines as those of Kawakita~\cite[Construction~3.1]{Kawakita}, we recall it for the reader's convenience.

\begin{enumerate}
\item We consider the blow-up $b'_i\colon \BL_{Z_{i-1}}X_{i-1}\to X_{i-1}$ and we take $b_i\colon X_i\to \BL_{Z_{i-1}}X_{i-1}$ a resolution of singularities. We put $g_i\coloneqq b'_i\circ b_i$. Since $Z_{i-1}$~is generically smooth, the exceptional divisor $E'_i$ of $b'_i$ is smooth over the generic point of $Z_{i-1}$, and so is its blow-up $\BL_{Z_{i-1}}X_{i-1}$. In particular, we can
find~$b_i$ such that none of the $b_i$-exceptional divisors surjects onto~$Z_{i-1}$.

\item We define $Z_i$ as the centre of $E$ in $X_i$, that is the closure of the image of~$E$ under the birational map~$g_i^{-1}\circ f$.  Note that $Z_i$~is reduced and generically smooth.

\item We denote by $E_i\subset X_i$ the strict transform of $E'_i$. Since $Z_i$ surjects onto~$Z_{i-1}$, the divisor $E_i$ is generically smooth along~$Z_i$, which is contained in~$E_i$ but in none of the $b_i$-exceptional divisors. Thus $E_i$ is the only  prime exceptional divisor of $g_i$ which contains $Z_i$, and furthermore $Z_i$ has multiplicity one along $E_i$.
\end{enumerate}
The construction is summarized in Diagram~\ref{diag:tower} and is illustrated in Figure~\ref{pic:tower} (see Appendix).

\begin{equation}
\label{diag:tower}
\xymatrix{
W\ar[ddd]^-{h'}\ar[rrr]^-{f'}&&& X_n\ar@{-->}@/_2pc/[dddlll]_-g\ar[d]^-{g_n}\ar[r]_-{b_n}\ar@/_4pc/[ddd]_-h\ar@/_2pc/[dd]_-{h_1}& \BL_{Z_{n-1}}X_{n-1}\ar[dl]^-{b'_n}\\
&&&X_{n-1}\ar@{..>}[d] \ar[r]_-{b_{n-1}}&\BL_{Z_{n-2}}X_{n-2}\ar@{..>}[dl] \\
&&&X_1\ar[r]_-{b_1}\ar[d]^-{g_1}&\BL_{Z_0} X_0\ar[dl]^-{b'_1}\\
Y\ar[rrr]^-f &&& X=X_0&}
\end{equation}

The process ends up at some $n$-th step, when $Z_n=E_n$ is a prime divisor.
Let us explain why it does so. At step $i\geq 1$, denote by $\crep(E_i,X_{i-1})\geq 1$ the discrepancy of~$E_i$ with respect to~$X_{i-1}$. Since~$Z_i$ has multiplicity one along~$E_i$, we have:
\[
g_{i+1}^\ast E_i\equiv (g^{-1}_{i+1})_\ast E_i+ E_{i+1}+\text{other components}.
\]
It is then easy to compute that:
\[
\crep(E_{i+1},X_{i-1})=\crep(E_{i+1},X_i)+\crep(E_i,X_{i-1}),
\]
and similarly, going down by induction, for any $1\leq j\leq i$:
\begin{equation}\label{eq:7}
\crep(E_{i+1},X_{i-j})\geq\crep(E_{i+1},X_i)+\crep(E_i,X_{i-j}),
\end{equation}
where the inequality comes from the possible contribution of $E_{i+1}$ to the total transforms of the birational images of the divisors $E_{i-1},\ldots,E_1$, for instance:
\begin{equation}\label{eq3}
g_{i+1}^\ast (g_i^{-1})_\ast E_{i-1}\equiv (g_{i+1}^{-1})_\ast (g_i^{-1})_\ast E_{i-1}+\mult_{Z_i}((g_i^{-1})_\ast E_{i-1})E_{i+1}.
\end{equation}
In particular, $\crep(E_{i},X_{i-j})<\crep(E_{i+1},X_{i-j})$ for any $j$. Moreover the first blow-up is the one of a centre of codimension three, so $\crep(E_1,X)= 2$ and:
\[
2= \crep(E_1,X)<\cdots<\crep(E_i,X)<\cdots<\crep(E_n,X)=\crep(E,X)=a.
\]
This shows that the process ends up after at most $a-1$ steps. We finally take an elimination of
 indeterminacies~$W$ of the birational map~$f^{-1}\circ h$.

We put:
\begin{align*}
h&\coloneqq g_1\circ\cdots\circ g_n,\\
h_i&\coloneqq g_{i+1}\circ\cdots\circ g_n,\quad\forall i=1,\ldots,n-1,\\
g_{i,j}&\coloneqq g_{j+1}\circ\cdots\circ g_i\colon X_i\to X_j, \quad\forall i>j\geq 0.
\end{align*}
For simplicity, in the sequel we still denote by~$E_j$ the divisor
$(g_{i,j}^{-1})_\ast E_j$ on~$X_i$. Recall that we denote by~$n\leq a-1$ the number of steps of the construction, which stops when $Z_n=E_n$ is a prime divisor. We also denote by $m\leq n$ the largest integer such that $\dim Z_{m-1}=\dim Z$.
Using the results recalled in \S\ref{s:equiv}, we see that we may assume that the tower construction~(\ref{diag:tower}) is  $G$-equivariant.

\begin{remark} Since $E$ and $E_n$ define the same valuation on $X$, we have:
\[
  f_\ast\cO_Y(-iE) \cong h_\ast\cO_{X_n}(-i E_n) \quad\forall i
\]
(see also~\cite[Remark~3.4]{Kawakita}).
By the same argument as above, it is easy to see that for any $i$, the possible contribution of $E_{i+1}$ to the total transforms of the birational images of the divisors $E_{i-1},\ldots,E_1$,
introduced in Equation~\eqref{eq3} have multiplicity at most one. Proposition~\ref{charactwbu} below says in particular that there is no such contribution
during the process exactly when $f_\ast\cO_Y(-2E)\neq\cI_Z$,
or equivalently $h_\ast\cO_{X_n}(-2E_n)\neq\cI_Z$.
\end{remark}

\subsection{Characterisation of the tower as a weighted blow-up}

Recall that if $f\colon Y\to X$ is a divisorial contraction with exceptional divisor $E$, the local ring $\cO_{Y,E}$ is a discrete valuation ring.
The basic observation to interpret the tower construction as a weighted blow-up is the
comparison of valuations (see \S\ref{ss:wbl}).

\begin{lemma}[\protect{\cite[Lemma~3.4]{Kawakita}}]
Two divisorial contractions over $X$ whose exceptional divisors define the same valuation are isomorphic.
\end{lemma}

This means that, to prove that the divisorial contraction ${f\colon Y\to X}$ of Theorem~\ref{th:main} is a $G$-equivariant weighted blow-up, it is enough to show that
the exceptional divisor $E_n$ obtained by the tower construction in \S\ref{ss:tower} defines the same valuation as those of a $G$-equivariant weighted blow-up over $X$.
To do this, one of the key results in Kawakita~\cite{Kawakita} is a characterisation of weighted blow-ups of smooth points in a threefold. We prove a similar result as \cite[Proposition 3.6]{Kawakita} in our setup.

\begin{proposition}[]\label{charactwbu}
Let $f\colon Y\to X$ be a $d$-dimensional $G$-equivariant divisorial contraction to a $(d-3)$-dimensional $G$-orbit contained in the smooth locus of $X$.
Keeping the same notation as above, the divisor $E_n$ defines the same valuation as those of an exceptional divisor obtained by a $G$-equivariant weighted blow-up of weights $(n,m,1)$
 if and only if the following
conditions hold for every  analytic open set $U$ in $X$:
{\rm \begin{enumerate}
 \item\label{charactwbu1} {\it $f_*\cO_Y(-2E)\vert_U\neq \cI_Z\vert_U$, that is $h_*\cO_{X_n}(-2E_n)\vert_U\neq \cI_Z\vert_U$;}
 \item \label{charactwbu2} {\it $f_*\cO_Y(-nE)\vert_U\not\subseteq \cI_Z^{2}\vert_U$, that is $h_*\cO_{X_n}(-nE_n)\vert_U\not\subseteq \cI_Z^{2}\vert_U$.}
\end{enumerate}}

\end{proposition}

\begin{proof}
Assume first that $f$ is a $G$-equivariant weighted blow-up of $Z$ with weights $(n,m,1)$. Let~$U$ be an analytic open set of~$X$, meeting~$Z$. Since $Z$~is smooth and contained in the smooth locus of~$X$, shrinking~$U$ if necessary, we may assume that there are local coordinates $x_1,\ldots,x_d$ on~$U$ such that $Z\cap U$ is given by the equations:
\[x_1=x_2=x_3=0,
\]
with the property that for any $k\geq 0$ we have:
\begin{equation}\label{eq4}
f_\ast\cO_Y(-kE)\vert_U=\cI_k\vert_U,
\end{equation}
where $\cI_k$ is the ideal sheaf generated over $U$ by the monomials $x_1^{a_1}x_2^{a_2}x_3^{a_3}$ with $na_1+ma_2+a_3\geq k$ (see \S\ref{ss:wbl}). Taking $k=2$, we see that $x_3\notin \cI_2\vert_U$, so $\cI_2\vert_U\neq\cI_Z\vert_U$: this is Condition~(\ref{charactwbu1}). Taking $k=n$, we see that $x_1\in\cI_n\vert_U$, but $x_1\notin\cI_Z^{2}\vert_U$, so $\cI_n\vert_U\not\subseteq \cI_Z^{2}\vert_U$: this is Condition~(\ref{charactwbu2}).

Conversely, assume now that Conditions (\ref{charactwbu1}) and (\ref{charactwbu2}) are satisfied. Let $U$ be an analytic open subset of $X$, meeting $Z$.
We may assume that $U$ is nonsingular. We have $f_\ast\cO_Y(-2E)\vert_U\subseteq \cI_Z\vert_U$.
During the tower construction, if there were a contribution of some divisor $E_{i+1}$ to the total transforms of the birational images of
the divisors $E_{i-1},\ldots,E_1$, that is, if there were an index $j\leq i$ such that $Z_j\subseteq E_{i+1}$, then at the end of the process,
the multiplicity of~$E_n$ in~$h_1^\ast E_1$ would be greater of equal than two (see Equation~\eqref{eq3}). This would mean that for any $\varphi\in\cI_Z\vert_U$, the function $\varphi\circ h$ vanishes at least twice along $E_n$, so:
\[
\cI_Z\vert_U\subset h_\ast\cO_{X_n}(-2E_n)\vert_U=f_\ast\cO_Y(-2E)\vert_U.
\]
This would contradict Condition~(\ref{charactwbu1}), so we deduce that there is no such contribution and we have the formulas:
\begin{equation}\label{eq2}
g_{i,j}^\ast E_j=\sum_{k=j}^i E_k+\text{other components},\quad 0\leq j<i\leq n.
\end{equation}
The sheaf $\left(\cI_Z/\cI_Z^{2}\right)\vert_U$ is locally free of rank~$3$ on~$U\cap Z$ and is locally generated by three coordinate functions of~$U$ (see~\cite[Theorem~8.17]{Hartshorne}). We have:
\[
h_*\cO_{X_n}(-nE_n)\vert_U\subseteq h_*\cO_{X_n}(-E_n)\vert_U\subseteq \cI_Z\vert_U,
\]
and by Condition~(\ref{charactwbu2}), we can choose an element $x_1\in\left(\cI_Z\setminus\cI_Z^{}\right)\vert_U$ such that $h_n^\ast\Div(x_1)\geq n E_n$. The function $x_1$ has multiplicity one along $Z\cap U$, this means that, shrinking $U$ if necessary, we may assume that $x_1$ is part of a coordinate system of $Z\cap U$.
At step $n$ of the tower construction, the process ends up when the centre $Z_n$ is the prime divisor $E_n$ in $X_n$.
We take a coordinate $x_3 \in  \left(\cI_Z/\cI_Z^{2}\right)\vert_U$ such that $x_3\circ h$ vanishes with multiplicity one along each $E_i$, shrinking again $U$ if necessary.
At step $m$ of the construction, the centre $Z_m$ has codimension at most two,
so, shrinking once more $U$ we choose a third coordinate function $x_2$ such that until step $m$, the local equations of $Z_i$ are the strict transforms of $x_1$ and $x_2$ and a local equation of $E_i$,
that is:
\begin{align*}
\sum\limits_{i=1}^n E_i&=\Div(x_3\circ h),\\
Z_i&=E_i\cap\Div(x_1)\cap\Div(x_2),\quad\forall i=1,\ldots,m-1,\\
Z_i&=E_i\cap\Div(x_1),\quad\forall i=m,\ldots,n-1,
\end{align*}
where, as explained above, we still denote by~$D$, instead of $(g_{i,0}^{-1})_\ast D$, the strict transform in~$X_j$ of a divisor~$D$ in~$X_0$. We finally take complementary coordinate functions $x_4,\ldots,x_d$ on $U$.
Denote by $v_{E_n}$ the valuation associated to the discrete valuation ring $\cO_{X_n,E_n}$. Locally on $U$, it is characterized by its value on the coordinate functions $x_1,\ldots,x_d$. Since $x_1$ vanishes along $Z_i$ for all $i=1,\ldots,n$, using Equation~\eqref{eq2} we have by induction:
\[
g_{i,0}^\ast \Div(x_1)=\Div(x_1)+\sum_{j=1}^i j E_j +\text{other components},\quad\forall i=1,\ldots,n.
\]
In particular, for $i=n$ we get $v_{E_n}(x_1)=n$. Similarly, since $x_2$ vanishes along~$Z_i$ only for $i=1,\ldots,m-1$, we compute:
\begin{align*}
g_{m,0}^\ast \Div(x_2)&=\Div(x_2)+\sum_{j=1}^m j E_j+\text{other components},\\
g_{m+1,0}^\ast \Div(x_2)&=\Div(x_2)+\sum_{j=1}^m j E_j+mE_{m+1}+\text{other components},\\
g_{n,0}^\ast \Div(x_2)&=\Div(x_2)+\sum_{j=1}^m j E_j+m(E_{m+1}+\cdots+E_n)+\text{other components}.
\end{align*}
This shows that $v_{E_n}(x_2)=m$. Clearly $v_{E_n}(x_3)=1$ and $v_{E_n}(x_i)=0$ for $i\geq 4$, so locally over $U$, $E_n$~defines the same valuation as a weighted blow-up with weights~$(n,m,1)$. Since the divisorial contraction~$f$ is $G$-equivariant, and since $Z$~is a $G$-orbit, a local analysis at another analytic neighbourhood of~$Z$ can be performed by translating the above local coordinate functions by the action of~$G$, so the weighted blow-up is given by the same weights everywhere and $f$~is a $G$-equivariant weighted blow-up with weights $(n,m,1)$.
\end{proof}

\begin{remark}\label{rem:no_accident}
If there is no contribution of $E_{i+1}$ to the total transform of the birational images of the divisors $E_{i-1},\ldots,E_1$
in Equation~\eqref{eq3}, for any $i$, then Equation~\eqref{eq:7} shows that the discrepancy function $i\mapsto\crep(E_i,X)$ increases by~$2$ until $i=m$ and by~$1$ after, so:
\[
a=\crep(E,X)=2m+(n-m)=n+m.
\]
This can be also directly computed from Equations~\eqref{eq:8} of a weighted blow-up.
\end{remark}


\section{Proof of Theorem~\ref{th:main}}\label{s:proof}

The proof of Theorem~\ref{th:main} uses an induction argument by considering a general hyperplane section of a locally $\IQ$-factorial variety. The next three lemmata contain the necessary ingredients for this induction. For any variety $X$, we denote by $\CL(X)$ the group of linear equivalence classes $[D]$ of Weil divisors $D$ on $X$.

\begin{lemma} \label{lem:Qfact}
Let $X$ be a normal, locally $\IQ$-factorial and projective variety of dimension at least four,
with terminal singularities. Let $L$ be a base point free line bundle on~$X$.
{\rm \begin{enumerate}
\item\label{lem:Qfact:1}  {\it A general element $H\in |L|$ is normal  with terminal singularities.}
\item\label{lem:Qfact:2} {\it If moreover $L$ is ample, then $H$ is also locally $\IQ$-factorial.}
\end{enumerate}}
\end{lemma}

\begin{proof}

By the Bertini theorem for irreductibility~\cite{Jouanolou}, the Seidenberg theorem~\cite{Seidenberg}  and the generalized Seidenberg theorem~\cite[Theorem~1.7.1]{BS} and its proof, a general element $H\in|L|$ is an irreducible Cartier divisor of $X$,  which is normal with terminal singularities in codimension at least three, proving~\eqref{lem:Qfact:1}.
As for \eqref{lem:Qfact:2}, by the Grothendieck--Lefschetz theorem for normal varieties, proven by Ravindra--Srinivas \cite[Theorem~1]{RS}, for $H$ general the restriction map $\CL(X)\to\CL(H)$ defined by $[D]\mapsto [D\cap H]$ is an isomorphism, so $H$ is locally $\IQ$-factorial.
\end{proof}

The following lemma is a consequence of \cite[Theorem~2]{RS}.

\begin{lemma}\label{lem:RS}
Let~$W$ be a normal, irreducible, projective variety of dimension greater or equal to four, $L$~a big line bundle on~$W$ and $V\subseteq \HH^0(W,L)$ linear subspace giving a
base point free linear system~$|V|$ on~$W$. Let $\varphi\colon W\to \IP(V^\ast)$ be the morphism determined by~$V$ and consider the Stein factorization of $\varphi$:
\[
\varphi\colon W\overset{\pi}{\longrightarrow}W'\longrightarrow \IP(V^\ast).
\]
Then there exists a Zariski open subset of divisors $Z \in |V|$ such that the cokernel of the restriction map:
\[
\rho\colon \CL(W)\to \CL(Z),
\]
is generated by the classes of $\mathcal O_Z(E)$, where $E$ is supported on $\Exc(\pi)\cap Z$.
\end{lemma}

\begin{proof}
Let $\varepsilon\colon \widehat W\to W$ be a resolution of singularities given by a sequence of blow-ups $\varepsilon=\varepsilon_s\circ\cdots\circ\varepsilon_1$. Since $Z$ is general, the morphism $\varepsilon$ induces a resolution of singularities $\varepsilon\vert_{\widehat Z}\colon \widehat Z\to Z$ such that no centre of a blow-up $\varepsilon_i$ is contained in the strict transform of $Z$.
Let $[F]\in \CL(Z)$ and let $\widehat F$ be its strict transform in~$\widehat Z$.
By \cite[Theorem~2]{RS} there are a divisor $\widehat D$ on $\widehat W$ and divisors $F_i$ contained in $\Exc(\pi\circ\varepsilon\vert_{\widehat Z})$ such that:
\[
[\widehat F]=[\widehat D\vert_{\widehat Z}]+\sum_i m_i[F_i].
\]
Then $[F]=\varepsilon_*[\widehat F]=\varepsilon_*[\widehat D]\vert_Z+\sum_i m_i \varepsilon_*[F_i]$, so the cokernel of the restriction map~$\rho$ is generated by the classes of exceptional divisors of~$\pi$.
\end{proof}

\begin{lemma}\label{lem:normal} Let $f\colon Y\to X$ be a $G$-equivariant divisorial contraction to a smooth $G$-simply connected $G$-orbit $Z$. Then general fibers of $f$ over $Z$ are irreducible.
\end{lemma}

\begin{proof}
Denote by $\nu\colon \widetilde E\to E$ the normalization of $E$. The morphism $\nu$ is birational and finite and the action of $G$ on~$E$ naturally extends to $\widetilde E$, so $\nu$ is $G$-equivariant.
Consider the restricted morphism $f_{|E}\colon E\to Z$ and the Stein factorization of the morphism $f_{|E}\circ\nu\colon \widetilde E\to Z$:
\[
\xymatrix{\widetilde E\ar[rr]^{f_{|E}\circ\nu}\ar[rd]^{\nu'} && Z\\
& Z'\ar[ur]^\eta}
\]
By construction, the fibration $\nu'$  and the finite morphism~$\eta$ are  $G$-equivariant.
Since $\eta$~is finite, surjective and $G$-equivariant, it is \'etale, otherwise it would be ramified everywhere since $G$ acts transitively on $Z$.
Since $Z$ is $G$-simply connected, $\eta$~is an isomorphism, so $f_{|E}\circ\nu=\nu'$ has connected fibers. Since $\widetilde E$ is normal, by the generalized Seidenberg theorem~\cite[Theorem~1.7.1]{BS},
a general fiber of $f_{|E}\circ\nu=\nu'$ is normal, hence irreducible since it is connected. It follows that a general fiber of $f_{|E}$ is irreducible.
\end{proof}

\begin{proof}[Proof of Theorem~\ref{th:main}]
Let $f\colon Y\to X$ be a $G$-equivariant divisorial contraction. We assume that its exceptional prime divisor is contracted to a smooth $(d-3)$-dimensional centre~$Z$ contained in the smooth locus of $X$, and that
$Z$ is a $G$-simply connected $G$-orbit. If $d=3$, the result follows from Kawakita's theorem (see Theorem~\ref{th:kawakita}) so we assume that $d\geq 4$.
By Lemma \ref{lem:Qfact}(\ref{lem:Qfact:1}), a general complete intersection of $(d-3)$ hyperplanes of $X$, that we denote by:
\[
H\coloneqq H_1\cap\cdots\cap H_{d-3}
\]
 is such that $H$~and $\widetilde H\coloneqq f^{-1}(H)$ are reduced,
irreducible, normal with terminal singularities.
Moreover, by Lemma \ref{lem:Qfact}(\ref{lem:Qfact:2}) the variety~$H$ is locally $\IQ$-factorial.
We claim:
\begin{claim}\label{claim:Qfact}
The variety~$\widetilde H$ is locally $\IQ$-factorial.
\end{claim}

Assuming the claim, we finish the proof.
Since the centre $Z$ of $f$ has codimension three in $X$, for $H$ general the intersection $H\cap Z$ is reduced and
consists of a finite number of points, so the restricted morphism
$f_{\mid\widetilde H}\colon \widetilde{H}\to H$ is birational and its exceptional locus is~$E\cap\widetilde H$. By Lemma~\ref{lem:normal}, the fibers of $f_{|\widetilde H}$ over $Z$ are irreducible.
Notice that $f^*H_k=\widetilde H_k$ for all $k$ where $\widetilde H_k$ is the strict transform of $H_k$ in $Y$.

By the adjunction formula we have:
\[
K_{\widetilde H}\equiv\left.\left(K_Y+\widetilde H_1+\cdots+\widetilde H_{d-3}\right)\right|_{\widetilde H}\equiv \left.\left(K_Y+f^* H_1+\cdots+f^*H_{d-3}\right)\right|_{\widetilde H}.
\]
Since $K_Y$ is $f$-antiample by assumption, we deduce that $K_{\widetilde H}$
is $f_{|\widetilde H}$-antiample.

To show that $f_{|\widetilde H}$ is a divisorial contraction over every point of $H\cap Z$,  we run on $\widetilde H$ a $K_{\widetilde H}$-MMP with scaling of an ample divisor over $H$. This MMP terminates by Lemma~\ref{easyterm} applied to~$(\widetilde H, 0)$:
\[
\xymatrix{\widetilde H \eqqcolon H_0\ar@{-->}[r]^{\varphi_0} \ar[drr]_{f_{|\widetilde H}\eqqcolon f_0} & H_1
\ar@{-->}[r]^{\varphi_1} \ar[dr]^{f_1}& \cdots\ar@{-->}[r]^{\varphi_{p-2}}& H_{p-1}\ar@{-->}[r]^{\varphi_{p-1}}\ar[dl]_{f_{p-1}} & H_p\ar[dll]^{f_p}\\
& & H}
\]

In this diagram, each dotted arrow $\varphi_i$ is either a divisorial contraction or a flip.
Since $f_0$ is birational, all structural morphisms $f_i$ are birational too.
The variety~$H_p$ is a minimal model relatively to~$H$, so it is normal, locally $\IQ$-factorial, with terminal singularities, and its canonical divisor is $f_p$-nef. Assume that $f_p$ is not an isomorphism. Since  $H$ has terminal singularities, there exist effective divisors~$F_i$ and positive rational numbers~$b_i$ such that:
\[
K_{H_p}\equiv f_p^\ast K_H+\sum_i b_i F_i.
\]
putting $D\coloneqq \sum_i b_i F_i$, we see that $D$~is $f_p$-nef, so by the negativity lemma~\cite[Lemma~3.39]{KollarMori}, $(-D)$~is effective, this is a contradiction. So $f_p$~is an isomorphism  and we may assume that the last step of the MMP is:
\[
\xymatrix{H_0\ar@{-->}[r]^{\varphi_0} \ar[drr]_{f_0} & \cdots \ar@{-->}[r]^{\varphi_{p-2}}
& H_{p-1}\ar@{-->}[r]^{\varphi_{p-1}}\ar[d]_{f_{p-1}} & H\ar@{=}[dl]\\
& & H}
\]

It follows that the birational map $\varphi_{p-1}$ is everywhere defined, so it is a divisorial contraction, with exceptional divisor denoted $D_{p-1}\subset H_{p-1}$.
Its strict transform $\widetilde D_{p-1}\subset \widetilde H$ is an irreducible component of the exceptional locus $E\cap\widetilde H$, so by Lemma~\ref{lem:normal} it is one of its connected components.
This means that, locally over~$H$ in a neighbourhood of $\varphi_{p-1}(D_{p-1})$ for the Zariski topology, the morphism~$f_{|\widetilde H}$ looks like a divisorial contraction.
More precisely, let $j$~be the maximum integer such that the exceptional locus of~$\varphi_j^{-1}$ is not disjoint from~$D_{p-1}$.
Then $\varphi_j$~cannot be a flip, otherwise $D_{p-1}$~would contain a $K_{H_{p-1}}$-positive curve, which is impossible since $K_{H_{p-1}}$~is $f_{p-1}$-negative.
Moreover $\varphi_j$~cannot be a divisorial contraction, otherwise there would be at least two irreducible components in $f_{|\widetilde H}^{-1}(\varphi_{p-1}(D_{p-1}))$.
Therefore the exceptional locus of~$\varphi_j^{-1}$ is disjoint from~$D_{p-1}$ for all~$j\leq p-2$.
In particular, $f_{p-1}$~commutes with~$\varphi_j$ for all~$j$ and we
can repeat the argument to show that all~$\varphi_i$ are divisorial contractions, whose exceptional divisors are the connected components of $E\cap\widetilde H$.
This shows that $f_{|\widetilde H}\colon\widetilde H\to H$ is a composition of divisorial contractions with disjoint exceptional loci.

To conclude, we prove that $f$ is a $G$-equivariant weighted blow-up by proving that it satisfies Conditions (\ref{charactwbu1})~and~(\ref{charactwbu2}) of Proposition~\ref{charactwbu}.
Indeed, assume by contradiction that (\ref{charactwbu1})~or~(\ref{charactwbu2}) is not satisfied. Then there is a point $p\in Z$ such that:
\begin{align}
\text{either } f_*\cO_Y(-2E)\otimes \cO_p&= \cI_Z\otimes \cO_p, \label{eq5}\\
\text{or }
 f_*\cO_Y(-nE)\otimes \cO_p&\subseteq \cI_Z^{2}\otimes \cO_p.\label{eq6}
\end{align}
Translating $H$ using the action of $G$ if necessary, we may assume that $p\in H$. Then locally around $p$, by Theorem~\ref{th:kawakita}, the divisorial contraction $f\vert_{\widetilde H}$ is a weighted blow-up.
So by Proposition \ref{charactwbu} applied to $f\vert_{\widetilde H}$, we have:
\begin{align*}
\left(f\vert_{\widetilde H}\right)_\ast\cO_{\widetilde H}\left(-2(E\cap \widetilde H)\right)&\neq \cI_p,\\
 \text{and }
\left(f\vert_{\widetilde H}\right)_\ast\cO_{\widetilde H}\left(-n(E\cap \widetilde H)\right)&\not\subseteq \cI_p^{2}.
\end{align*}
Tensoring Formulas~\eqref{eq5} and~\eqref{eq6} by $\cO_H$, we get:
\begin{align*}
\text{either }\left(f\vert_{\widetilde H}\right)_\ast\cO_{\widetilde H}\left(-2(E\cap \widetilde H)\right)&= \cI_Z\otimes \cO_p\otimes\cO_H,\\ \text{ or }
 \left(f\vert_{\widetilde H}\right)_\ast\cO_{\widetilde H}\left(-n(E\cap \widetilde H)\right)&\subseteq \cI_Z^{2}\otimes \cO_p\otimes\cO_H.
\end{align*}
But by generality of $H$, we have $\cI_Z\otimes \cO_p\otimes\cO_H=\cI_p$: this gives a contradiction, so $f$ is a $G$-equivariant weighted blow-up.
We are left with the proof of Claim \ref{claim:Qfact}. We prove by induction on $k$ that
$\widetilde H_1\cap\cdots\cap\widetilde H_k$ is locally $\IQ$-factorial for any $k\geq 0$. For $k=0$ there is nothing to prove, so assume that $k>0$. The divisor:
\[
L\coloneqq\widetilde H_k\vert_{\widetilde H_1\cap\cdots\cap \widetilde H_{k-1}}
\]
is big and base point free.
We apply Lemma~\ref{lem:RS} with:
\begin{align*}
W&\coloneqq\widetilde H_1\cap\cdots\cap \widetilde H_{k-1},&
|V|&\coloneqq |L|,\\
Z&\coloneqq\widetilde H_1\cap\cdots\cap\widetilde H_k,&
\pi&=f\vert_{W}\colon W\to f(W).
\end{align*}
We get that for any $[D]\in \CL(Z)$, either $[D]$~is in the image of the restriction map $\rho\colon\CL(W)\to\CL(Z)$, and therefore is $\IQ$-Cartier by induction, or, modulo an element of the
image of the restriction map, we can assume that $D$~is supported on the exceptional locus of the restriction of $\pi$.
So we have to prove that the irreducible components of the exceptional locus of $\pi$ are $\IQ$-Cartier divisors.
Denote by $(E_i)_{i=1,\ldots,s}$ the irreducible components of
$E\cap \widetilde Z$.
By Lemma \ref{lem:normal}, these are also the connected components of $E\cap Z$ (the only case where the exceptional locus is not connected in general is for $k=d-3$). We thus have a decomposition
into  connected irreducible components:
\[
E\cap Z=E\cap \widetilde H_1\cap\cdots\cap\widetilde H_k=E_1\cup\cdots\cup E_s.
\]
Then each $E_i$ is $\IQ$-Cartier, since it is the intersection of the $\IQ$-Cartier divisor~$E$ with:
\[
Z\setminus (E_1\cup\cdots \cup E_{i-1}\cup E_{i+1}\cup\cdots\cup E_s).
\]
\end{proof}

\begin{remark} \textit{A posteriori}, our induction argument, reducing to Kawakita's theorem~\ref{th:kawakita} shows that the arithmetic between the integer parameters in issue still holds in the setup of Theorem~\ref{th:main}:
\begin{enumerate}
\item The integers $n$ and $m$ are coprime (see~\cite[Theorem~3.5]{Kawakita});
\item The integers $a$ is prime to the index of $K_Y$ (see~\cite[Lemma~4.3]{Kawakita});
\item $a\geq 2$ (see~\cite[\S 5, p.116]{Kawakita}).
\end{enumerate}
\end{remark}

\section{A counter-example}
\label{s:counterexample}

In this section, to emphasize the importance of the group action in the statement of Theorem~\ref{th:main}, we construct a non equivariant $4$-dimensional divisorial contraction which is not a weighted blow-up. We use similar notation as in the tower construction in Section~\ref{s:tower}.

Let $X_0 \coloneqq X$ be a smooth projective variety of dimension $4$ and $Z_0 \coloneqq Z \subset X$ be a smooth curve.
\begin{enumerate}
\item We denote by  $g_1\colon X_1\to X$  the blow-up of $Z$, with exceptional divisor~$E_1$.
Let $Z_1 \subset X_1$ be a section of the $\IP^2$-bundle $\left.g_1\right|_{E_1} \colon E_1 \to Z$.

\item We denote by $g_2\colon X_2\to X_1$ the blow-up of $Z_1$, with exceptional divisor~$E_2$.
Again $\left.g_1\circ g_2\right|_{E_2}\colon E_2\to Z$ is a $\IP^2$-bundle.
We denote by $E_1^2$ the strict transform of $E_1$ in $X_2$. Observe that $\left.g_1\circ g_2\right|_{E_1^2}\colon E_1^2\to Z$ is an $\IF_1$-bundle. The natural fibration $\IF_1 \to \IP^1$ defines a $\IP^1$-bundle structure $\pi\colon E_1^2\to S$ over a surface $S$.

\item Let $Z_2\subseteq E_2$ be a surface such that the intersection of $Z_2$ with a general fibre of $\left.g_1\circ g_2\right|_{E_2}\colon E_2\to Z$ is a line and the intersection of $Z_2$ with $E_1^2$ is of the form
$s_2\cup h_1\cup\ldots\cup h_k$ where $s_2$ is a section of $\left.g_1\circ g_2\right|_{E_1^2}\colon E_1^2\to Z$,
and $h_i=(g_1\circ g_2)^{-1}(p_i)\cap E_1^2\cap E_2$ for some point $p_i\in Z$.
We can further assume that $h_i$ and $s_2$ meet transversally for all $i$.
We denote by $g_3\colon X_3\to X_2$ the blow up of $Z_2$, with exceptional divisor~$E_3$. We denote by $E_1^3$ the strict transform of $E_1^2$ in $X_3$.
\end{enumerate}
The construction is summarized in Diagram~\ref{diag:tower_counter_example} and is illustrated in Figure~\ref{pic:tower_counter_example} (Appendix). When there is no risk of confusion, we simply denote by $E_j$, instead of~$E_j^i$, the strict transform of $E_j$ in $X_i$ for $i=2, 3$ and $j=1, 2$.

\begin{proposition}\label{prop:logres}
The pair $(X_3, E_1)$ admits a log resolution $\mu\colon \widetilde X\to X_3$ such that:
\[
K_{\widetilde X}+(\mu^{-1})_\ast E_1=\mu^*(K_{X_3}+E_1)+\sum_i  F_i,
\]
with $\cup_i F_i=\Exc(\mu)$. The pair $(X_3, t E_1)$ is \klt for any $t\in\mathopen]0,1\mathclose[$.
\end{proposition}

\begin{proof}
The singularities of $E_1^3$ lie over the singular points of $s_2\cup h_1\cup\ldots\cup h_k$.
To understand them, we can choose local analytic coordinates $(x,y,z,w)$ on $X_2$ such that $Z_2$ is locally given by
$\{w=z=0\}$
and $E^2_1$ by $\{xy-w=0\}$.
The blow up of~$X_2$ along $Z_2$ is then locally given by:
\[
    \{((x,y,z,w),[s:t])\in \IC^4\times \IP^1\,|\, zt-ws=0\}.
\]
We compute that the variety $E_1^3$ has one singular point in the chart $s\neq 0$, with equation $zv-xy$, where $v=t/s$. It is smooth on the other chart.
A desingularization of $E_1^3$ is given by the blow up $\mu\colon\widetilde X\to X_3$ of the singular points of $E_1^3$.
Since $E_1^3$ has isolated singularities of type $A_1$, a local computation shows that the exceptional divisors $F_i$ of $\mu$ meet the strict transform $(\mu^{-1})_\ast E_1^3$ transversally, so $\mu$~is a log resolution and we have:
\begin{align*}
\mu^\ast E_1^3 &= (\mu^{-1})_\ast E_1^3+2\sum_i F_i,\\
K_{\widetilde X} &= \mu^\ast K_{X_3}+3\sum_i F_i.
\end{align*}
It follows that:
\[
K_{\widetilde X}+(\mu^{-1})_\ast E_1^3=\mu^*(K_{X_3}+E_1^3)+\sum_i  F_i,
\]
and that for any $t\in\mathopen]0,1\mathclose[$ we have:
\begin{align*}
K_{\widetilde X} + (\mu^{-1})_\ast (t E_1^3)
&=\mu^*(K_{X_3}+tE_1^3)+ (3-2t)\sum_i  F_i,
\end{align*}
proving that the pair $(X_3, E_1)$ is \klt.
\end{proof}

Over a general point of $Z$, the fiber of $E_1^3$ is isomorphic to the blow-up of $\IF_1$ at a point not belonging to a $(-1)$-curve. Over the points $p_i$, the fiber of $E_1^3$ is the union of the blow-up of $\IF_1$ at a point and a $\IP^1$-bundle over $\IP^1$.
Consider the composition $\pi'\colon E_1^3 \to E_1^2 \xrightarrow{\pi} S$. We denote by $\ell'_1$~the strict transform of a fibre of~$\pi$ passing through~$s_2$ and by $\ell''_1$~the strict transform  by $\pi'$ of  a fibre of~$\pi$ passing through~$h_i$.

\begin{lemma}\label{lem:num_equiv}
The curves $\ell'_1$ and $\ell''_1$ are numerically equivalent and generate an extremal ray of the cone $\overline{\NE}(X_3/X)$. More precisely, we have:
\begin{align*}
\ell'_1\cdot E_1&=\ell''_1\cdot E_1=-2,\\
\ell'_1\cdot E_2&=\ell''_1\cdot E_2=0,\\
\ell'_1\cdot E_3&=\ell''_1\cdot E_3=1,\\
\ell'_1\cdot K_{X_3}&=\ell''_1\cdot K_{X_3}=1.
\end{align*}
\end{lemma}
\begin{proof}
To prove the numerical equivalence, it is enough to prove that:
\[
\ell'_1\cdot E_i=\ell''_1\cdot E_i \text{ for } i=1,2,3.
\]
\begin{itemize}
\item We have $\ell'_1\cdot E_2=\ell''_1\cdot E_2=0$ as both $\ell'_1$ and $\ell''_1$ are disjoint from $E_2$.
\item
We have $\ell'_1\cdot E_3=\ell''_1\cdot E_3=1$ as both $\ell'_1$ and $\ell''_1$ meet $E_3$ transversally.
\item
We have $g_2^\ast E_1 = E_1^2+E_2$, $g_3^\ast E_1^2 = E_1^3$, and $g_3^\ast E_2^2 = E_2^3 + E_3$.
Set $\bar \ell'_1=g_{3 *}\ell'_1$ and $\bar \ell''_1=g_{3 *}\ell''_1$. The curves $\bar \ell'_1$ and $\bar \ell''_1$ are numerically equivalent as they are both fibres of $\pi\colon E_1^2\to S$. We get the numerical equivalence by computing as follows:
\[
E_1^3\cdot \ell'_1
=g_3^*E_1^2\cdot \ell'_1
=E_1^2\cdot \bar \ell'_1
=E_1^2\cdot \bar \ell''_1
=g_3^*E_1^2\cdot \ell''_1
=E_1^3\cdot \ell''_1.
\]
Moreover, $E_1^2\cdot \bar \ell'_1=(\varepsilon_2^*E_1-E_2)\cdot \bar \ell'_1=-2$.
Using the formulas:
\[
K_{X_1} = g_1^\ast K_X + 2 E_1,
\quad K_{X_2} = g_2^\ast K_{X_1} + 2 E_2,
\quad K_{X_3} = g_3^\ast K_{X_2} + E_3,
\]
we get $K_{X_3}=(g_1 g_2 g_3)^*K_X+2E_1+4E_2+5E_3$,
hence $K_{X_3}\cdot \ell'_1=1$.
\end{itemize}

Let us show that the ray $\IR_+[\ell'_1]$ is extremal. Assume $\ell'_1=C_1+C_2$ in $\overline{\NE}(X_3/Z)$. We have to show that $C_1,C_2\in \IR_+[\ell'_1]$. As $\ell'_1\cdot E_1=-2$, we can assume that every component of $C_2$ is linearly equivalent to a component not in $E_1$ and all the components of $C_1$ are in $E_1$. We have 
\[\overline{\NE}(E_1/Z)=\IR_+[\ell'_1]+\IR_+[\ell''_1]+\IR_+[\ell_2]+\IR_+[\ell_3],\]
therefore:
\begin{align*}
C_1\equiv& a \ell'_1+b\ell''_1+c\ell_2+d\ell_3\quad\mathrm{and}\\
\ell'_1\equiv& a \ell'_1+b\ell''_1+c\ell_2+d\ell_3+C_2,
\end{align*}
where the first numerical equivalence is in $E_1$ and the second in $X_3$.
We already proved that $\ell'_1\equiv \ell''_1$, hence we get:
\[
(1-a-b)\ell'_1\equiv c\ell_2+d\ell_3+C_2.
\]
As the right hand side is an effective curve, we have $1-a-b\geq 0$.
By taking the intersection product with $E_1$, we get
$-2(1-a-b)=c+C_2\cdot E_1\geq 0$. Therefore $a+b=1$ and $c\ell_2+d\ell_3+C_2\equiv 0$. This proves that $\IR_+[\ell'_1]$ is an extremal ray. The proof for $\IR_+[\ell''_1]$ is similar.
\end{proof}

\begin{lemma}\label{antiflip}
The antiflip
$\phi\colon X_3\dasharrow Y_3$ of the ray $\IR_+[\ell'_1]$ exists. The variety~$Y_3$ has terminal singularities.
\end{lemma}

\begin{proof}
In order to prove that the antiflip exists, it is enough to find a \klt boundary~$\Delta$ such that $(K_{X_3}+\Delta)\cdot \ell'_1<0$.
We set $\Delta=tE_1$ with $1/2<t<1$. The pair $(X_3, \Delta)$ is \klt by Proposition~\ref{prop:logres}, and using Lemma~\ref{lem:num_equiv} we get:
\[
(K_{X_3}+tE_1)\cdot \ell'_1=1-2t<0.\]
Then the $(K_{X_3}+tE_1)$-flip of $\ell'_1$ exists by~\cite[Corollary 1.4.1]{HMcK10}, we denote it by $\phi\colon X_3\dasharrow Y_3$.
By Proposition~\ref{prop:logres} the pair $(X_3,\Delta)$ is terminal  in the sense of \cite[Definition~2.34]{KollarMori}, so by
 \cite[Corollary~3.42]{KollarMori} the pair $(Y_3,\Delta^+)$ is also terminal. By ~\cite[Corollary 2.35]{KollarMori} the variety $Y_3$ has terminal singularities.
\end{proof}

We notice that $\phi$ is an isomorphism in a neighborhood of $E_2$.
We denote by $\overline E_1$, $\overline E_2$, $\overline E_3$ the strict transforms of $E_1$, $E_2$, $E_3$ in $Y_3$. In the next lemmata, we construct extremal contractions $Y_3\to Y_2\to Y_1$.

\begin{lemma}\label{contr2}
Let $\ell_2$ be the strict transform in $\overline E_2$ of a line in a fibre of $E_2\to Z$.
The extremal ray $\IR_+[\ell_2]$ is $K_{Y_3}$-negative in $\overline{\NE}(Y_3/X)$. We denote by  $\eta_3\colon Y_3\to Y_2$ the contraction of $\IR_+[\ell_2]$. We have:
\[
K_{Y_3}=\eta_3^*K_{Y_2}+\frac{1}{2} \overline E_2.
\]
\end{lemma}

\begin{proof}
As $E_2$ is disjoint from the indeterminacy locus of $\phi$, we have $K_{Y_3}\cdot \ell_2=K_{X_3}\cdot \ell_2=-1$.
Writing $K_{Y_3}=\eta_3^*K_{Y_2}+a \overline E_2$, we have $a=\frac{1}{2}$ since:
\[
-1=K_{Y_3}\cdot \ell_2=\eta_3^*K_{Y_2}\cdot \ell_2+a \overline E_2\cdot \ell_2=-2a.
\]
\end{proof}

\begin{lemma}\label{contr1}
Let $\ell_1\subset Y_2$ be the strict transform in $\overline E_1$ of a fibre of $\pi\colon E_1\to Z$ not passing through $s_2\cup h_1\cup\ldots\cup h_k$. Then $\IR_+[\ell_1]$ is a $K_{Y_2}$-negative extremal ray in $\overline{\NE}(Y_2/X)$.
\end{lemma}

\begin{proof}
By abuse of notation we still denote by $\ell_1$ the strict transform of $\ell_1$ in $Y_3$.
By Lemma~\ref{contr2}, we obtain:
\[
K_{Y_2}\cdot \ell_1
= K_{Y_3}\cdot \ell_1- \frac{1}{2} \overline E_2\cdot \ell_1\\
= K_{X_3}\cdot \ell_1-\frac{1}{2}  E_2\cdot \ell_1
=-\frac{1}{2}.
\]
\end{proof}

We denote by  $\eta_2\colon Y_2\to Y_1\eqqcolon Y$ the contraction of $\IR_+[\ell_1]$

\begin{proposition}
There is a divisorial contraction $f\colon \overline X\to X$ with  codimension three center contained in the smooth locus. The divisorial contraction $f$ is not a weighted blow-up.
\end{proposition}

\begin{proof}
Starting from the variety $Y_3$ obtained in Lemma~\ref{antiflip} by flipping over $X$ the extremal ray $\IR_+[l'_1]$, we contract the divisor $\overline E_2$ by Lemma~\ref{contr2}, obtaining the variety $Y_2$. By Lemma~\ref{contr1} we can contract the divisor $\overline E_1$ over $X$, obtaining a variety $Y$ with a birational morphism  $f\colon Y\to X$. The variety $Y$ has terminal singularities by Lemma~\ref{antiflip} and because it is obtained by contracting $K$-negative rays of the Mori cone. Furthermore $\rho(Y)=\rho(X_3)-2=\rho(X)+1$ so $f$ is a divisorial contraction.
We set $E$ the exceptional divisor of $f$, which is the strict transform of $\overline E_3$.
The construction is summarized in Diagram~\ref{diag:tower_counter_example}.

\begin{equation}\label{diag:tower_counter_example}
\xymatrix{
Y_3 \ar@{-->}[r]^{\phi^{-1}}\ar[d]^{\eta_3}                  & X_3 \ar[d]^{g_3} \\
Y_2 \ar[d]^{\eta_2}                  & X_2 \ar[d]^{g_2} \\
Y\coloneqq Y_1 \ar[dr]_f & X_1 \ar[d]^{g_1} \\
                         & X
}
\end{equation}

We prove now that $f$ is not a weighted blow-up.
As $g_1$ and $g_2$ blow-up curves and $g_3$ a surface, by the characterization of the tower construction given in Section~\ref{s:tower} the fibre of $f$ over a general point of $Z$ is the weighted projective plane $\IP(3,2,1)$, which has three singular points.

We prove that over the special points there is either a curve of singularities or a unique singular point. From this it follows that $f$ cannot be a weighted blow-up, otherwise all the fibres of the restriction $f\colon E\to Z$ would be isomorphic.

Let $p\in Z$ be a special point, that is, one of the images of the singularities of $s_2\cup h_1\cup\ldots\cup h_k$. Let $H\subseteq X$ be a smooth hyperplane through $p$ and $S\subseteq H$ a smooth surface through $p$.
By choosing $H, S$ general enough, we can assume that  the centre $Z_i$ is not contained in the strict transforms of $S$ in $X_i$ and that the strict transform of $S$ in $X_3$ is smooth.
We can find such $S$ as $X$ is smooth. Let $S_3$ be the strict transform of $S$ in $X_3$.
Let $\bar g$ be the restriction of $g_3\circ g_2\circ g_1$ to $S_3$. Then $g^{-1}(p)$ is a
 chain of curves $e_1,e_2,e_3$ where $e_3$ meets $e_2$ and $e_1$,
$E_2$ marks on $S_3$ the $(-2)$-curve $e_2$, $E_3$ the $(-1)$-curve $e_3$ and $E_1$ the union $e_3\cup e_1$ and $e_1$ is a $(-3)$-curve.

Let $\overline S_3$ be the strict transform of $S_3$ via the antiflip and $\lambda\colon S_3\to \overline S$ and $\mu\colon \overline S_3\to \overline S$ be the restriction of the antiflip to~$S_3$. Then  $\overline E_1\cup \overline E_2\cup \overline E_3$ marks on $\overline S_3$ a chain of curves $\overline e_2,\overline e_3,\overline \ell_3$ where $\overline e_3$ meets $\overline e_2$ and $\overline \ell_3$, and $\overline e_3$ and $\overline e_2$ are the strict transforms of $e_3$ and $e_2$.
From the adjunction formula it follows that $\overline \ell_3$ has positive intersection with the canonical bundle of $\overline S_3$. Indeed, if we let $\overline H_3$ be the strict transform of $H$ in $Y_3$ by our hypothesis $\overline H_3=\eta_3^*\eta_2^*f^*H$, and we have $K_{\overline H_3}=K_{Y_3}+\overline H_3$. Moreover, again by the generality of $S$, $\overline S_3$ is the pullback of $S$ via the restriction of $f\circ\eta_3\circ\eta_2$. By the adjunction formula we have
$$K_{\overline S_3}\cdot \overline\ell_3=(K_{\overline H_3}+\overline S_3)\cdot \overline\ell_3=K_{\overline H_3}\cdot \overline\ell_3=K_{Y_3}\cdot \overline\ell_3>0.$$
Therefore there is a singular point for $\overline S_3$ in $\overline \ell_3$ : otherwise $\overline S_3$ would be the minimal desingularisation of the surface obtained by contracting $e_1$, but this surface is isomorphic to $S_3$ and the exceptional curve is not $K$-negative.

Two cases may appear. Either for $S$~general the singular point lies away from~$\overline  e_3$,
in which case the fibre of  $f\colon E\to Z$ over~$p$ is  singular along a curve, or the singular point is the intersection of~$\overline  e_3$ and~$\overline  \ell_3$, and this curve gets contracted to a point via the MMP, in which case the fibre of  $f\colon E\to Z$ over $p$ has just one singular point.
\end{proof}


\section*{Appendix. Graphical illustrations of the key geometric constructions of the paper}
\addcontentsline{toc}{section}{Appendix. Graphical illustrations of
the key geometric constructions of the paper}

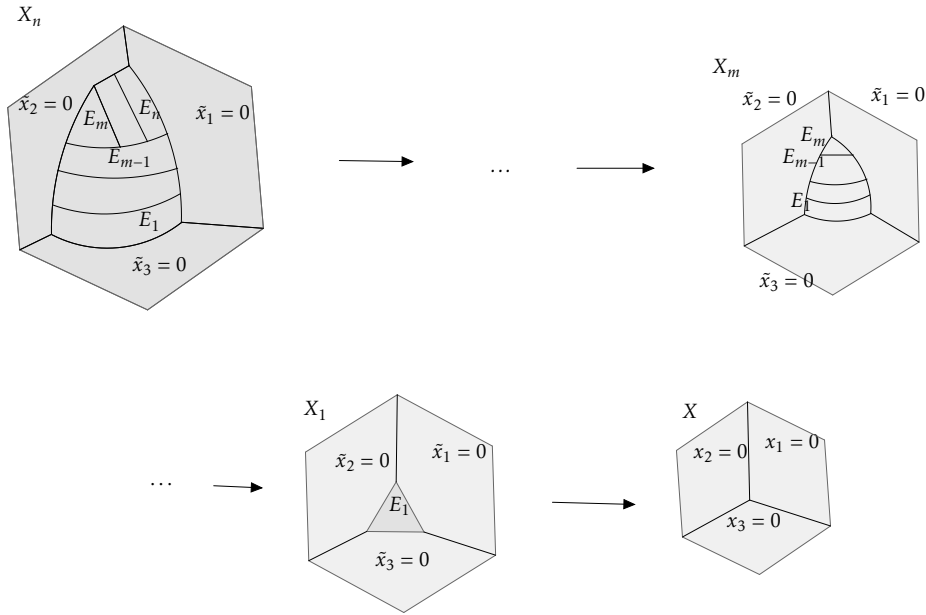
\begin{figure}[!ht]
\centering
\hspace*{-2cm}
\definecolor{wwwwww}{rgb}{0.4,0.4,0.4}
\resizebox{15cm}{12cm}{
\begin{tikzpicture}[line cap=round,line join=round,>=triangle 45,x=1.0cm,y=1.0cm]
\clip(-5.23,-11.09) rectangle (15.29,5.33);
\fill[color=wwwwww,fill=wwwwww,fill opacity=0.1] (-0.8,2.62) -- (-2.92,1.14) -- (-2.7,-1.44) -- (-0.36,-2.53) -- (1.76,-1.05) -- (1.54,1.52) -- cycle;
\fill[color=wwwwww,fill=wwwwww,fill opacity=0.1] (12.1,1.44) -- (10.52,0.48) -- (10.56,-1.37) -- (12.18,-2.26) -- (13.76,-1.3) -- (13.72,0.55) -- cycle;
\fill[color=wwwwww,fill=wwwwww,fill opacity=0.1] (-0.8,2.62) -- (-2.92,1.14) -- (-2.7,-1.44) -- (-0.36,-2.53) -- (1.76,-1.05) -- (1.54,1.52) -- cycle;
\fill[color=wwwwww,fill=wwwwww,fill opacity=0.1] (4.21,-4.07) -- (2.53,-5.11) -- (2.59,-7.08) -- (4.33,-8.02) -- (6.01,-6.98) -- (5.95,-5) -- cycle;
\fill[color=wwwwww,fill=wwwwww,fill opacity=0.1] (4.19,-5.65) -- (3.65,-6.55) -- (4.7,-6.56) -- cycle;
\fill[color=wwwwww,fill=wwwwww,fill opacity=0.1] (10.62,-4.2) -- (9.32,-5.08) -- (9.43,-6.65) -- (10.84,-7.33) -- (12.14,-6.45) -- (12.03,-4.89) -- cycle;
\draw [color=wwwwww] (-0.8,2.62)-- (-2.92,1.14);
\draw [color=wwwwww] (-2.92,1.14)-- (-2.7,-1.44);
\draw [color=wwwwww] (-2.7,-1.44)-- (-0.36,-2.53);
\draw [color=wwwwww] (-0.36,-2.53)-- (1.76,-1.05);
\draw [color=wwwwww] (1.76,-1.05)-- (1.54,1.52);
\draw [color=wwwwww] (1.54,1.52)-- (-0.8,2.62);
\draw [shift={(2.61,-1.06)}] plot[domain=2.56:3.16,variable=\t]({1*4.73*cos(\t r)+0*4.73*sin(\t r)},{0*4.73*cos(\t r)+1*4.73*sin(\t r)});
\draw [shift={(-1.1,0.81)}] plot[domain=4.24:5.37,variable=\t]({1*2.22*cos(\t r)+0*2.22*sin(\t r)},{0*2.22*cos(\t r)+1*2.22*sin(\t r)});
\draw (-0.7,1.9)-- (-1.34,1.54);
\draw [shift={(-4.08,-0.82)}] plot[domain=-0.03:0.68,variable=\t]({1*4.34*cos(\t r)+0*4.34*sin(\t r)},{0*4.34*cos(\t r)+1*4.34*sin(\t r)});
\draw [shift={(-1.16,3.17)}] plot[domain=4.46:5.15,variable=\t]({1*2.76*cos(\t r)+0*2.76*sin(\t r)},{0*2.76*cos(\t r)+1*2.76*sin(\t r)});
\draw (-1.34,1.54)-- (-0.86,0.42);
\draw [shift={(-0.97,3.7)}] plot[domain=4.44:5.02,variable=\t]({1*3.84*cos(\t r)+0*3.84*sin(\t r)},{0*3.84*cos(\t r)+1*3.84*sin(\t r)});
\draw [shift={(-1.15,2.02)}] plot[domain=4.37:5.23,variable=\t]({1*2.81*cos(\t r)+0*2.81*sin(\t r)},{0*2.81*cos(\t r)+1*2.81*sin(\t r)});
\draw (-0.97,1.75)-- (-0.36,0.52);
\draw (-0.7,1.9)-- (-0.8,2.62);
\draw (-2.12,-1.16)-- (-2.7,-1.44);
\draw (0.26,-0.94)-- (1.76,-1.05);
\draw [->] (3.15,0.18) -- (4.51,0.16);
\draw [->] (0.85,-5.72) -- (1.75,-5.76);
\draw [->] (7.05,-6.02) -- (8.57,-6.06);
\draw (10.8,-4.64) node[anchor=north west] {$x_1=0$};
\draw (9.49,-4.82) node[anchor=north west] {$x_2=0$};
\draw (10.1,-6.07) node[anchor=north west] {$x_3=0$};
\draw (0.38,1.32) node[anchor=north west] {$\tilde x_1=0$};
\draw (-2.86,1.49) node[anchor=north west] {$\tilde x_2=0$};
\draw (-0.78,-1.43) node[anchor=north west] {$\tilde x_3=0$};
\draw (-0.66,-0.62) node[anchor=north west] {$E_1$};
\draw (-1.68,1.22) node[anchor=north west] {$E_m$};
\draw (-0.66,1.41) node[anchor=north west] {$E_n$};
\draw (4.71,-4.82) node[anchor=north west] {$\tilde x_1=0$};
\draw (2.94,-4.99) node[anchor=north west] {$\tilde x_2=0$};
\draw (3.67,-6.78) node[anchor=north west] {$\tilde x_3=0$};
\draw (3.92,-5.78) node[anchor=north west] {$E_1$};
\draw [color=wwwwww] (12.1,1.44)-- (10.52,0.48);
\draw [color=wwwwww] (10.52,0.48)-- (10.56,-1.37);
\draw [color=wwwwww] (10.56,-1.37)-- (12.18,-2.26);
\draw [color=wwwwww] (12.18,-2.26)-- (13.76,-1.3);
\draw [color=wwwwww] (13.76,-1.3)-- (13.72,0.55);
\draw [color=wwwwww] (13.72,0.55)-- (12.1,1.44);
\draw [shift={(14.3,-0.93)}] plot[domain=2.51:3.09,variable=\t]({1*2.64*cos(\t r)+0*2.64*sin(\t r)},{0*2.64*cos(\t r)+1*2.64*sin(\t r)});
\draw [shift={(11.4,-0.66)}] plot[domain=-0.08:1.03,variable=\t]({1*1.48*cos(\t r)+0*1.48*sin(\t r)},{0*1.48*cos(\t r)+1*1.48*sin(\t r)});
\draw [shift={(12.25,0.58)}] plot[domain=4.31:5.15,variable=\t]({1*1.5*cos(\t r)+0*1.5*sin(\t r)},{0*1.5*cos(\t r)+1*1.5*sin(\t r)});
\draw (11.66,-0.8)-- (10.56,-1.37);
\draw (12.88,-0.78)-- (13.76,-1.3);
\draw (12.1,1.44)-- (12.16,0.62);
\draw [shift={(12.22,1.47)}] plot[domain=4.45:5.06,variable=\t]({1*1.73*cos(\t r)+0*1.73*sin(\t r)},{0*1.73*cos(\t r)+1*1.73*sin(\t r)});
\draw (11.95,0.28)-- (12.56,0.28);
\draw [shift={(12.22,0.8)}] plot[domain=4.33:5.19,variable=\t]({1*1.4*cos(\t r)+0*1.4*sin(\t r)},{0*1.4*cos(\t r)+1*1.4*sin(\t r)});
\draw (12.74,1.65) node[anchor=north west] {$\tilde x_1=0$};
\draw (10.39,1.57) node[anchor=north west] {$\tilde x_2=0$};
\draw (10.7,-1.72) node[anchor=north west] {$\tilde x_3=0$};
\draw (11.27,-0.27) node[anchor=north west] {$E_1$};
\draw (11.45,0.92) node[anchor=north west] {$E_m$};
\draw [color=wwwwww] (-0.8,2.62)-- (-2.92,1.14);
\draw [color=wwwwww] (-2.92,1.14)-- (-2.7,-1.44);
\draw [color=wwwwww] (-2.7,-1.44)-- (-0.36,-2.53);
\draw [color=wwwwww] (-0.36,-2.53)-- (1.76,-1.05);
\draw [color=wwwwww] (1.76,-1.05)-- (1.54,1.52);
\draw [color=wwwwww] (1.54,1.52)-- (-0.8,2.62);
\draw [shift={(2.61,-1.06)}] plot[domain=2.56:3.16,variable=\t]({1*4.73*cos(\t r)+0*4.73*sin(\t r)},{0*4.73*cos(\t r)+1*4.73*sin(\t r)});
\draw [shift={(-1.1,0.81)}] plot[domain=4.24:5.37,variable=\t]({1*2.22*cos(\t r)+0*2.22*sin(\t r)},{0*2.22*cos(\t r)+1*2.22*sin(\t r)});
\draw (-0.7,1.9)-- (-1.34,1.54);
\draw (-1.34,1.54)-- (-0.86,0.42);
\draw (-0.7,1.9)-- (-0.8,2.62);
\draw (-2.12,-1.16)-- (-2.7,-1.44);
\draw (0.26,-0.94)-- (1.76,-1.05);
\draw (5.79,0.15) node[anchor=north west] {$\ldots$};
\draw [->] (7.5,0.04) -- (8.92,0.04);
\draw [color=wwwwww] (4.21,-4.07)-- (2.53,-5.11);
\draw [color=wwwwww] (2.53,-5.11)-- (2.59,-7.08);
\draw [color=wwwwww] (2.59,-7.08)-- (4.33,-8.02);
\draw [color=wwwwww] (4.33,-8.02)-- (6.01,-6.98);
\draw [color=wwwwww] (6.01,-6.98)-- (5.95,-5);
\draw [color=wwwwww] (5.95,-5)-- (4.21,-4.07);
\draw [color=wwwwww] (4.19,-5.65)-- (3.65,-6.55);
\draw [color=wwwwww] (3.65,-6.55)-- (4.7,-6.56);
\draw [color=wwwwww] (4.7,-6.56)-- (4.19,-5.65);
\draw (3.65,-6.55)-- (2.59,-7.08);
\draw (4.7,-6.56)-- (6.01,-6.98);
\draw (4.21,-4.07)-- (4.19,-5.65);
\draw (-0.56,-5.53) node[anchor=north west] {$$ \ldots $$};
\draw [color=wwwwww] (10.62,-4.2)-- (9.32,-5.08);
\draw [color=wwwwww] (9.32,-5.08)-- (9.43,-6.65);
\draw [color=wwwwww] (9.43,-6.65)-- (10.84,-7.33);
\draw [color=wwwwww] (10.84,-7.33)-- (12.14,-6.45);
\draw [color=wwwwww] (12.14,-6.45)-- (12.03,-4.89);
\draw [color=wwwwww] (12.03,-4.89)-- (10.62,-4.2);
\draw (10.66,-5.98)-- (9.43,-6.65);
\draw (10.66,-5.98)-- (10.62,-4.2);
\draw (10.66,-5.98)-- (12.14,-6.45);
\draw (-2.88,3.1) node[anchor=north west] {$X_n$};
\draw (-1.23,0.53) node[anchor=north west] {$E_{m-1}$};
\draw (11.12,0.53) node[anchor=north west] {$E_{m-1}$};
\draw (9.84,2.16) node[anchor=north west] {$X_m$};
\draw (2.36,-4.05) node[anchor=north west] {$X_1$};
\draw (9.29,-4.09) node[anchor=north west] {$X$};
\end{tikzpicture}
}
\vspace{-1.5cm}
\caption{The tower construction}
\label{pic:tower}
\end{figure}

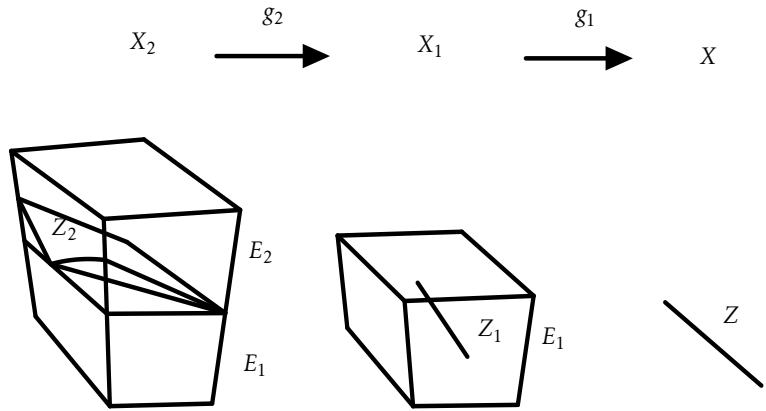
\begin{figure}[!ht]
\centering
\resizebox{13cm}{9cm}{
\begin{tikzpicture}[line cap=round,line join=round,>=triangle 45,x=1cm,y=1cm]
\clip(-7.378773445127541,-5.92286509426191) rectangle (8.50224090660893,4.6644778068957375);
\draw [line width=2pt] (5,-2)-- (6.58,-3.3);
\draw [line width=2pt] (0.86,-3.62)-- (2.58,-3.6);
\draw [line width=2pt] (2.58,-3.6)-- (2.84,-1.9);
\draw [line width=2pt] (0.86,-3.62)-- (0.72,-1.98);
\draw [line width=2pt] (0.72,-1.98)-- (2.84,-1.9);
\draw [line width=2pt] (0.72,-1.98)-- (-0.38,-0.96);
\draw [line width=2pt] (-0.38,-0.96)-- (-0.22,-2.4);
\draw [line width=2pt] (-0.22,-2.4)-- (0.86,-3.62);
\draw [line width=2pt] (-0.38,-0.96)-- (1.66,-0.9);
\draw [line width=2pt] (1.66,-0.9)-- (2.84,-1.9);
\draw [line width=2pt] (-4.12,-3.62)-- (-2.44,-3.58);
\draw [line width=2pt] (-2.44,-3.58)-- (-1.98,-0.56);
\draw [line width=2pt] (-1.98,-0.56)-- (-4.22,-0.7);
\draw [line width=2pt] (-4.22,-0.7)-- (-4.12,-3.62);
\draw [line width=2pt] (-4.12,-3.62)-- (-5.34,-2.22);
\draw [line width=2pt] (-5.34,-2.22)-- (-5.74,0.36);
\draw [line width=2pt] (-5.74,0.36)-- (-4.22,-0.7);
\draw [line width=2pt] (-5.74,0.36)-- (-3.56,0.54);
\draw [line width=2pt] (-3.56,0.54)-- (-1.98,-0.56);
\draw [line width=2pt] (-5.52287659174931,-1.0404459832169475)-- (-4.169327585398997,-2.179634506349281);
\draw [line width=2pt] (-4.169327585398997,-2.179634506349281)-- (-2.224984997856837,-2.1683797685383626);
\draw [line width=2pt] (-2.224984997856837,-2.1683797685383626)-- (-4.198061009324774,-1.3406185277166018);
\draw [shift={(-4.471809900710919,-3.775548702283563)},line width=2pt]  plot[domain=1.4588406706424029:1.8259331294435601,variable=\t]({1*2.450270068900944*cos(\t r)+0*2.450270068900944*sin(\t r)},{0*2.450270068900944*cos(\t r)+1*2.450270068900944*sin(\t r)});
\draw [line width=2pt] (-5.09020363378183,-1.404596867798)-- (-5.6251475852356085,-0.3807980752303268);
\draw [line width=2pt] (-5.6251475852356085,-0.3807980752303268)-- (-3.84,-1.06);
\draw [line width=2pt] (-3.84,-1.06)-- (-2.224984997856837,-2.1683797685383626);
\draw [line width=2pt] (-5.09020363378183,-1.404596867798)-- (-2.224984997856837,-2.1683797685383626);
\draw (5.819877185006908,-1.9259655090933354) node[anchor=north west] {$Z$};
\draw (5.429069225568203,2.106462072387849) node[anchor=north west] {$X$};
\draw (2.835525494747706,-2.263481474063126) node[anchor=north west] {$E_1$};
\draw [line width=2pt] (0.9347776920230949,-1.6950335330613733)-- (1.7519216072131147,-2.8496934132211837);
\draw (1.7874496035257244,-2.139133486968993) node[anchor=north west] {$Z_1$};
\draw (0.7926657067726567,2.3018660521072016) node[anchor=north west] {$X_1$};
\draw (-2.0673379963924123,-2.672053431658136) node[anchor=north west] {$E_1$};
\draw (-1.9785180056108886,-0.9134176141839631) node[anchor=north west] {$E_2$};
\draw (-5.211565670058358,-0.5581376510578675) node[anchor=north west] {$Z_2$};
\draw (-3.932557802804414,2.3373940484198115) node[anchor=north west] {$X_2$};
\draw [->,line width=2pt] (-2.369325965049594,1.8400021000432776) -- (-0.5041061586375921,1.8400021000432776);
\draw [->,line width=2pt] (2.7111775076535727,1.804474103730668) -- (4.469813325127745,1.804474103730668);
\draw (3.386209437593154,2.6926740115459067) node[anchor=north west] {$g_1$};
\draw (-1.7475860295789265,2.7282020078585165) node[anchor=north west] {$g_2$};
\end{tikzpicture}
}
\vspace{-1.5cm}
\caption{The counter-example}
\label{pic:tower_counter_example}
\end{figure}

\newpage




\end{document}